\documentclass[11pt,a4paper,english,reqno]{amsart}
\usepackage[a4paper,footskip=1.5em,bottom=38mm,margin=2.34cm]{geometry}
\usepackage{amsmath,amssymb,amsthm,mathtools}
\usepackage[mathscr]{euscript} \usepackage[usenames,dvipsnames]{color}
\usepackage{adjustbox,tikz,calc,graphics,babel,standalone}
\usepackage{subcaption} \usepackage{csquotes,enumerate,verbatim}
\usepackage[final]{microtype} \usepackage[numbers]{natbib}
\usepackage{soul}
\usepackage{appendix}
\usetikzlibrary{shapes.misc,calc,intersections,patterns,decorations.pathreplacing}
\usepackage{hyperref}
\hypersetup{colorlinks=true,linkcolor=blue,citecolor=blue,pdfpagemode=UseNone,pdfstartview={XYZ
null null 1.00}} 
\usepackage[none]{hyphenat}
\theoremstyle{plain}
\newtheorem*{theorem*}{Theorem}
\newtheorem{theorem}{Theorem}[section]

\newtheorem{lemma}[theorem]{Lemma}
\newtheorem{claim}[theorem]{Claim}

\newtheorem*{claim*}{Claim}
\newtheorem{corollary}[theorem]{Corollary}
\newtheorem{conjecture}[theorem]{Conjecture}

\theoremstyle{remark}

\DeclareMathOperator\ex{ex}

\let\emptyset\varnothing
\let\eps\varepsilon

\title{Uniform chain decompositions and applications}
\thanks{ETH Zurich, \emph{e-mail}: \textbf{\{benjamin.sudakov,istvan.tomon,zsolt.wagner\}@math.ethz.ch}. Research supported by SNSF grant 200021-149111.}
\author{
Benny Sudakov 
\and
Istv\'an Tomon 
\and
Adam Zsolt Wagner
}

\begin{document}

\maketitle

\begin{abstract}
    The Boolean lattice $2^{[n]}$ is the family of all subsets of $[n]=\{1,\dots,n\}$ ordered by inclusion, and a chain is a family of pairwise comparable elements of $2^{[n]}$. Let $s=2^{n}/\binom{n}{\lfloor n/2\rfloor}$, which is the average size of a chain in a minimal chain decomposition of $2^{[n]}$. We prove that $2^{[n]}$ can be partitioned into $\binom{n}{\lfloor n/2\rfloor}$ chains such that all but at most $o(1)$ proportion of the chains have size $s(1+o(1))$. This asymptotically proves a conjecture of F\"uredi from 1985. Our proof is based on probabilistic arguments. To analyze our random partition
    we develop a weighted variant of the graph container method. 
     
    Using this result, we also answer a Kalai-type question raised recently by Das, Lamaison and Tran. What is the minimum number of forbidden comparable pairs forcing that the largest subfamily of $2^{[n]}$ not containing any of them has size at most $\binom{n}{\lfloor n/2\rfloor}$? We show that the answer is $(\sqrt{\frac{\pi}{8}}+o(1))2^{n}\sqrt{n}$. 
     
     Finally, we discuss how these uniform chain decompositions can be used to optimize and simplify various results in extremal set theory.
\end{abstract}

\section{Introduction}

The \emph{Boolean lattice $2^{[n]}$} is the family of all subsets of $[n]=\{1,\dots,n\}$, ordered by inclusion. A \emph{chain} in $2^{[n]}$ is a family $\{x_{1},\dots,x_{k}\}\subset 2^{[n]}$ such that $x_{1}\subset \dots\subset x_{k}$, and an \emph{antichain} is a family $A\subset 2^{[n]}$ such that no two elements of $A$ are comparable. 

A cornerstone result in extremal set theory is the theorem of Sperner~\cite{S28} which states that the size of the largest antichain in $2^{[n]}$ is $\binom{n}{\lfloor n/2\rfloor}$, which by Dilworth's theorem~\cite{D50} is equivalent to the statement that the minimum number of chains $2^{[n]}$ can be partitioned into is also $\binom{n}{\lfloor n/2\rfloor}$. While the maximum sized antichain is more or less unique (if $n$ is odd, there are two maximal antichains, otherwise it is unique), there are many different ways to partition $2^{[n]}$ into the minimum number of chains. In general, chain decompositions of the Boolean lattice into the minimum number of chains are extensively studied, see e.g. \cite{BTK51,DJMS19,F85,GK76,G88,HLST02,HLST03, T15, T16}.

One minimal chain decomposition of particular interest is the so-called \emph{symmetric chain decomposition}. A chain with elements $x_{0}\subset\dots\subset x_{k}$ is \emph{symmetric} in $2^{[n]}$, if $|x_{i}|=\frac{n-k}{2}+i$ for $i=0,\dots,k$. It was proved by de Brujin, Tengbergen and Kruyswijk~\cite{BTK51} that the Boolean lattice can be partitioned into symmetric chains. Note that in such a chain decomposition, there are exactly $\binom{n}{k}-\binom{n}{k-1}$ chains of size $n-2k+1$ for $k=0,\dots,\lfloor n/2\rfloor$. Therefore, in a symmetric chain decomposition the sizes of the chains are distributed very non-uniformly, in fact, it is the most non-uniform chain decomposition in a certain sense, see the discussion in Section \ref{sect:remarks}. Perhaps motivated by this observation, F\"uredi~\cite{F85} asked whether there exists a chain decomposition of $2^{[n]}$ into the minimum number of chains such that any two chains have roughly the same size.

\begin{conjecture}[F\"{u}redi~\cite{F85}]\label{conj:mainconj}
Let $n$ be a positive integer and let $s=2^{n}/\binom{n}{\lfloor n/2\rfloor}.$ Then $2^{[n]}$ can be partitioned into $\binom{n}{\lfloor n/2\rfloor}$ chains such that the size of each chain is either $\lfloor s\rfloor$ or $\lceil s\rceil$.
\end{conjecture}

Here, we have $s=(\sqrt{\frac{\pi}{2}}+o(1))\sqrt{n}\approx 1.25\sqrt{n}$. Hsu, Logan, Shahriari and Towse~\cite{HLST02,HLST03} proved the existence of a chain decomposition into the minimum number of chains such that the size of each chain is between $\frac{1}{2}\sqrt{n}+O(1)$ and $O(\sqrt{n\log n})$. The second author of this paper~\cite{T15,T16} improved the lower and upper bound to $0.8\sqrt{n}$ and $26\sqrt{n}$, respectively, and proved certain generalizations of his result to other partially ordered sets. The main result of our paper is that Conjecture~\ref{conj:mainconj} holds asymptotically.

\begin{theorem}\label{thm:mainthm}
Let $n$ be a positive integer and $s=2^{n}/\binom{n}{\lfloor n/2\rfloor}$. The Boolean lattice can be partitioned into $\binom{n}{\lfloor n/2\rfloor}$ chains such that all but at most $n^{-\frac{1}{8}+o(1)}$ proportion of the chains have size ${s(1+O(n^{-\frac{1}{16}}))}$.
\end{theorem}

\noindent
We made no serious attempt to optimize the error terms in this result. Let us remark that we will show in Section 2.7 that the chain decomposition provided by Theorem~\ref{thm:mainthm} has the following additional property.

\begin{corollary}
\label{remark}
 Let $\mathcal{C}$ be a chain decomposition of $2^{[n]}$ provided by Theorem~\ref{thm:mainthm}. Then the chains of size ${s(1+O(n^{-\frac{1}{16}}))}$ in $\mathcal{C}$ cover $1-n^{-\frac{1}{8}+o(1)}$ proportion of $2^{[n]}$.
\end{corollary}

Our main theorem has the following interesting application. 
The well known theorem of Mantel states that if a graph $G$ with $n$ vertices does not contain a triangle, then it has at most $\lfloor \frac{n^{2}}{2}\rfloor$ edges, and this bound is sharp for every $n$. Kalai (see~\cite{DLT19}) proposed the following question: what is the size of the smallest set $T$ of triples in an $n$ element vertex set $V$ such that any graph on $V$ with $\lfloor \frac{n^{2}}{2}\rfloor+1$ edges contains a triangle spanned by a triple in $T$? Das, Lamaison and Tran~\cite{DLT19} proved that the answer is $(\frac{1}{2}+o(1))\binom{n}{3}$, where the upper bound also follows from an earlier work of Allen, B\"ottcher, Hladk\'y, Piguet~\cite{ABHP13}. The authors also propose to study Kalai-type questions for other well known extremal problems.  Motivated by Sperner's theorem they asked 
for the minimum number of forbidden comparable pairs forcing that the largest subfamily of $2^{[n]}$ not containing any of them has size at most $\binom{n}{\lfloor n/2\rfloor}$.
Let $B_{n}$ denote the comparability graph of $2^{[n]}$, that is, $V(B_{n})=2^{[n]}$ and $x,y\in 2^{[n]}$ are joined by an edge if $x\subset y$ or $y\subset x$. It is a nice exercise to show that $B_{n}$ has $3^{n}-2^{n}$ edges. Sperner's theorem is equivalent to the statement that the size of the largest independent set of $B_{n}$ is $\binom{n}{\lfloor n/2\rfloor}$. In this setting, the question of Das, Lamaison and Tran can be reformulated as follows. What is the least number of edges of a subgraph $G$ of $B_{n}$ with $V(G)=2^{[n]}$ such that $G$ has no independent set larger than $\binom{n}{\lfloor n/2\rfloor}$? Using Theorem~\ref{thm:mainthm}, we answer this question asymptotically.

\begin{theorem}\label{thm:sperner}
Let $G$ be a subgraph of $B_{n}$ with the minimum number of edges such that $V(G)=2^{[n]}$ and $G$ has no independent set larger than $\binom{n}{\lfloor n/2\rfloor}$. Then $|E(G)|=(\sqrt{\frac{\pi}{8}}+o(1))2^{n}\sqrt{n}$.
\end{theorem}

Finally, we show that the uniform chain decomposition provided by Theorem~\ref{thm:mainthm} can be applied to various extremal set theory problems, generalizing ideas of the second author~\cite{T19}. The typical question in extremal set theory is that how large can be a family $H\subset 2^{[n]}$ that avoids a certain forbidden configuration. One way to attack such a problem is as follows. A \emph{$d$-dimensional grid} is a $d$-term  Cartesian product of the form $[k_1]\times\dots\times[k_d]$. We fix some $d$ and partition $2^{[n]}$ into $d$-dimensional grids of roughly the same size. Then, we bound the size of the intersection of each of these grids with the family $H$ avoiding the forbidden configuration. The advantage of this approach is that the problem of the maximal subset of the grid avoiding a given forbidden configuration is equivalent to an (ordered) hypergraph Tur\'an problem, for which sometimes there is already an  available good bound. 
In order to find a partition into $d$-dimensional grids, we write $2^{[n]}$ as the Cartesian product $2^{[n_{1}]}\times\dots\times 2^{[n_{d}]}$, where $n_{i}\approx \frac{n}{d}$, and find a uniform chain decomposition $\mathcal{C}_{i}$ of $2^{[n_{i}]}$. Then the Cartesian products $C_{1}\times\dots\times C_{d}$, where $C_{1}\in\mathcal{C}_{1},\dots,C_{d}\in\mathcal{C}_{d}$, partition  $2^{[n]}$ in the desired manner. We will illustrate how to apply this idea in case when the forbidden configuration is two sets and their union, a copy of some poset $P$, or a full Boolean algebra.

Our paper is organized as follows. In Section~\ref{sect:mainproof}, we prove Theorem~\ref{thm:mainthm} and Corollary~\ref{remark}. In Section~\ref{sect:sperner}, we prove Theorem~\ref{thm:sperner}. In Section~\ref{sect:extremal}, we discuss further possible applications of our main result in extremal set theory.

\section{Decomposition into chains of uniform size}\label{sect:mainproof}

\subsection{Preliminaries}
We use the following standard graph theoretic notation. If $G$ is a graph and $x\in V(G)$, then $\deg_{G}(x)$ denotes the degree of $x$ in $G$. Also, if $U\subset V(G)$, then $N_{G}(U)=\{y\in V(G)\setminus U:\exists x\in U, xy\in E(G)\}$ is the \emph{external neighborhood} of $U$ in $G$, and if $U=\{x\}$, we write $N_{G}(x)$ instead of $N_{G}(\{x\})$.

Also, we use the following set theoretic notation. If $0\leq l\leq n$, then $[n]^{(l)}=\{x\in 2^{[n]}:|x|=l\}$ and $[n]^{(\geq l)}=\{x\in 2^{[n]}:|x|\geq l\}$. We define $[n]^{(\leq l)}$ similarly. Also, a \emph{level} of $2^{[n]}$ refers to one of the families $[n]^{(l)}$ for $l=0,\dots,n$.

The proof of our main theorem uses probabilistic tools, see the book of Alon and Spencer~\cite{AS04} for a general reference about the probabilistic method. In particular, we need the following variants of Chernoff's inequality, see e.g. Theorem 2.8 in \cite{JLR00}.

\begin{claim}\label{claim:chernoff}(Chernoff's inequality)
 Let $X_1,\dots,X_n$ be independent random variables such that $\mathbb{P}(X_{i}=1)=p_{i}$ and $\mathbb{P}(X_{i}=0)=1-p_{i}$, and let $X=\sum_{i=1}^{n}X_i$. Then for $\delta>0$, we have
 
$$\mathbb{P}(X\geq (1+\delta)\mathbb{E}(X))\leq 
\begin{cases}
e^{-\frac{\delta^{2}}{3}\mathbb{E}(X)} &\mbox{ if }\delta\leq 1,\\
e^{-\frac{\delta}{3}\mathbb{E}(X)} &\mbox{ if }\delta>1.
\end{cases}$$

Also, if $p_{1}=\dots=p_{n}=\frac{1}{2}$ and $t>0$, then
$$\mathbb{P}\left(X\geq \frac{n}{2}+t\right)\leq e^{-\frac{2t^{2}}{n}}.$$
\end{claim}

Our proof of Theorem~\ref{thm:mainthm} depends quite delicately on the distribution of the sizes of the levels of $2^{[n]}$. Next, we collect some estimates on the binomial coefficients we use in this paper.

\begin{claim}\label{claim:binomial}
Let $n$ be a positive integer, $m=\lceil \frac{n}{2}\rceil$ and $M=\binom{n}{m}$.
\begin{enumerate}
\item $M=\left(\sqrt{\frac{2}{\pi}}+o(1)\right)\frac{2^{n}}{\sqrt{n}}.$~\cite{S14}
\item For $l=o(n^{2/3})$, $\binom{n}{m+l}=(1+o(1))Me^{-2l^{2}/n}.$~\cite{S14}
\item For $0<l$, $\sum_{i>m+l}\binom{n}{i}\leq 2^{n}e^{-2l^{2}/n}.$ (Chernoff's inequality)
\item For $0<l<\sqrt{n}$, $M\left(1-\frac{2l^{2}}{n}\right)\leq \binom{n}{m+l}<M\left(1-\frac{l^{2}}{4n}\right)$.
\item For $0\leq l<10\sqrt{n}$, $\binom{n}{m+l}-\binom{n}{m+l+1}=\Theta(l2^{n}n^{-3/2}).$
\item For $\sqrt{n}\leq l=o(n^{2/3})$, $\sum_{i\geq m+l}\binom{n}{i}\geq (e^{-7}+o(1))2^ne^{-2l^2/n}\frac{\sqrt{n}}{l}$.
\end{enumerate}
\end{claim}

\begin{proof}
  See the Appendix.
\end{proof}

\subsection{Overview of the proof}

The proof of Theorem~\ref{thm:mainthm} is somewhat technical at certain stages, so let us roughly outline our strategy. Let $k=\lceil s/2\rceil$. First of all, we only consider the upper half of $2^{[n]}$, $B=[n]^{(\geq \lfloor n/2\rfloor)}$, as if we manage to partition $B$ into chains of size $k$ approximately, then we can easily turn it into a chain partition of $2^{[n]}$ with the desired properties.

We start with the $k$ largest levels.  The remaining levels $[n]^{(l)}$ for $l> \lceil n/2\rceil+k$ we cut into small pieces and glue these small pieces to the levels $[n]^{(\lceil n/2\rceil)},\dots,[n]^{(\lceil n/2\rceil+k)}$ such that every level of the resulting new poset has size roughly $\binom{n}{\lfloor n/2\rfloor}$. Since this new poset has exactly $k+1$ levels, one can hope to find a chain partition of it into $\binom{n}{\lfloor n/2\rfloor}$ chains, each of size $\approx k$. Indeed, we show that 
if we cut the levels $[n]^{(l)}$ for $l> \lceil n/2\rceil+k$ randomly, then such a chain partition exists with high probability.

\subsection{Setting up} 

Throughout this section, we assume that $n$ is sufficiently large for our arguments to work. Let $m=\lceil \frac{n}{2}\rceil$, $M=\binom{n}{m}$, $A_{i}=[n]^{(m+i)}$ for $i=0,\dots,n-m$, and $B=[n]^{(\geq m)}$. Then $|B|=2^{n-1}$ if $n$ is odd, and $|B|=2^{n-1}+\frac{M}{2}$ if $n$ is even. We remind the reader that $s=\frac{2^{n}}{M}$, and define $k=\lceil \frac{s}{2}\rceil$. Note that  $(k-1)M<|B|<(k+1)M$.  Also, as $s=(1+o(1))\sqrt{\pi/2}\sqrt{n}$, we have $|A_{k}|=Me^{-\pi/4+o(1)}$. In particular, $0.45 M< |A_{k}|<0.46 M$.

Consider the subposet $P_{0}$ of $B$ induced by the levels $A_{0},\dots,A_{k}$. Next, we would like to ''fill up'' $P_{0}$ with the elements of $[n]^{(>k+m)}$, that is, we want to add elements of $[n]^{(>k+m)}$ to the levels $A_{1},\dots,A_{k}$ such that the size of each level becomes roughly $M$. We do this as follows: imagine a $(k+1)\times M$ sized rectangle partitioned into $(k+1)M$ unit squares indexed by $(a,b)\in \{0,\dots,k\}\times [M]$, where we fill some of the unit squares with the elements of $B$. We want do this in a way such that each row corresponds to an expanded level $A_{i}'$. First, for $a=0,\dots,k$, fill the unit squares $(a,1),\dots, (a,|A_{a}|)$ with the elements of $A_{a}$. Then, we will fill the rest of the unit squares as follows. For $1\leq a\leq b\leq k$, let $X_{a,b}=\{b\}\times \{|A_{a}|+1,\dots,|A_{a-1}|\}$, and for $l=0,\dots,k$, let the $l$-th diagonal be the union $\bigcup_{l\leq a\leq k-l} X_{a,a+l}$. Note that $|X_{a,b}|=|A_{a-1}|-|A_{a}|$ and the size of the $l$-th diagonal is $M-|A_{k-l}|$. Order the elements of $[n]^{(>k+m)}$ in an increasing order of the sizes, and among sets of the same size, chose a random ordering. Start filling up the first diagonal using the elements of $[n]^{(>k+m)}$ with respect to this order. Then if the $l$-th diagonal is already filled up, we move to the $(l+1)$-th diagonal. Also, we fill up each diagonal from right to left. We do this until we run out of elements in $[n]^{(>k+m)}$. In the end, the $i$-th row of the rectangle becomes the level $A_{i}'$, and we get a poset $P$ with levels $A_{0}',\dots,A_{k}'$ in which $x\leq_{P} y$ if  $x$ and $y$ are in different levels and $x\subset y$. Then $P$ is a subposet of $B$ of height $k+1$ such that every level of $P$ has size roughly $M$. Our goal (more or less) is to show that $P$ can be partitioned into $M$ chains. In the rest of the proof, we shall not work directly with the poset $P$, but for a better understanding of our proof, it is worth seeing this underlying structure. See Figure \ref{figure:chains} for an illustration. 
 
 For the sake of clarity, let us define our sets $X_{a,b}$ formally. Let $C_{0}=\left\lceil \sqrt{\frac{1}{3}n \log n}\right\rceil$. Let $T=\bigcup_{k+1\leq i\leq C_{0}}A_{i}$ and $Z=[n]^{(>m+C_{0})}$. Then $|Z|\leq n^{-2/3}2^{n}$ by Claim~\ref{claim:binomial}, (3). For $ k+1\leq i\leq C_{0}$, let $\prec_{i}$ be a random total ordering on $A_{i}$ (chosen uniformly among all the total orders), and define the total ordering $\prec$ on $T$ such that for $x\in A_{a}$ and $y\in A_{b}$, we have $x\prec y$ if $a<b$, or $a=b$ and $x\prec_{a} y$. In other words, we randomly order the elements of the levels from $A_{k+1}$ to $A_{C_{0}}$, and then we lay out these levels next to each other, this is the total order $(T,\prec)$.
 
 Each set $X_{a,b}$ will be an interval in $T$ with respect to the total order $\prec$. Let $I^{*}=\{(a,b):1\leq a\leq b\leq k\}$, which will serve as the set of possible indices of these intervals. Order the elements of $I^{*}$ by $\prec'$ such that $(a,b)\prec' (a',b')$ if $b-a<b'-a'$, or $b-a=b'-a'$ and $a<a'$, then $\prec'$ will be the order of our desired intervals. Cut $T$ into intervals $X_{a,b}$, where $(a,b)\in I^{*}$, with the following procedure. Let $(1,1)=(a_{1},b_{1})\prec'\dots\prec' (a_{|I^{*}|},b_{|I*|})$ be the elements of $I^{*}$, and let $X_{a_{1},b_{1}}$ be the initial segment of $T$ of size $|A_{0}|-|A_{1}|$. Now if $X_{a_{l},b_{l}}$ is already defined for $l\geq 1$, and there are still at least $|A_{a_{l+1}-1}|-|A_{a_{l+1}}|$ elements of $T$ larger than $X_{a_{l},b_{l}}$ with respect to $\prec$, then let $X_{a_{l+1},b_{l+1}}$ be the $|A_{a_{l+1}-1}|-|A_{a_{l+1}}|$ smallest elements of $T$ larger than $X_{a_{l},b_{l}}$. Otherwise, stop, and set $I=\{(a_{j},b_{j}):1\leq j\leq l\}$. 

\begin{figure}
\begin{tikzpicture}
  \draw[dashed] (0,0) rectangle (10,2.1) ;
  \draw (0,0) rectangle (10,0.3) ; \node at (5,0.15) {\tiny $A_{0}$} ;
  \draw (0,0.3) rectangle (9.6,0.6) ; \node at (4.8,0.45) {\tiny $A_{1}$} ;
  \draw (0,0.6) rectangle (9,0.9) ;
  \draw (0,0.9) rectangle (8.2,1.2) ;  
  \draw (0,1.2) rectangle (7,1.5) ;
  \draw (0,1.5) rectangle (5.7,1.8) ;
  \draw (0,1.8) rectangle (4.5,2.1) ; \node at (2.25,1.95) {\tiny $A_k$} ;

  \draw[fill=black!10!white] (-2.1,3.5) rectangle (1.5,3.8) ; \node at (-0.3,3.97) {\tiny $A_{k+1}$} ;
  \draw[fill=black!20!white] (1.5,3.5) rectangle (4.5,3.8) ;   \node at (3,3.97) {\tiny $A_{k+2}$} ;
  \draw[fill=black!30!white] (4.5,3.5) rectangle (6.9,3.8) ;
  \draw[fill=black!40!white] (6.9,3.5) rectangle (8.9,3.8) ; 
  \draw[fill=black!50!white] (8.9,3.5) rectangle (10.6,3.8) ; 
  \draw[fill=black!60!white] (10.6,3.5) rectangle (12.1,3.8) ; \node at (11.35,3.97) {\tiny $A_{C_{0}}$} ;
 
  \draw[decoration={brace,mirror,raise=5pt},decorate]
  (-2.1,3.5) -- node[below=6pt] {\small $T$} (12.1,3.5);

  
  \draw[->,very thick] (6.5,3) -- (7,2.4) ;

  \draw[fill=black!10!white] (9.6,0.3) rectangle (10,0.6) ; \node at (10.5,0.45) {\tiny $X_{1,1}$} ; \draw[->] (10.2,0.45) -- (9.8,0.45) ;
  \draw[fill=black!10!white] (9,0.6) rectangle (9.6,0.9) ; \node at (10.5,0.75) {\tiny $X_{1,2}$} ; 
  \draw[fill=black!10!white] (8.2,0.9) rectangle (9,1.2) ;
  \draw[fill=black!10!white] (7,1.2) rectangle (8.2,1.5) ;
  \draw[fill=black!20!white] (5.7,1.5) rectangle (6.4,1.8) ;  \draw[fill=black!10!white] (6.4,1.5) rectangle (7,1.8) ;
  \draw[fill=black!20!white] (4.5,1.8) rectangle (5.7,2.1) ; \node at (5.1,1.95) {\tiny $X_{k,k}$} ;

  \draw[fill=black!20!white] (9.6,0.6) rectangle (10,0.9) ;
  \draw[fill=black!20!white] (9,0.9) rectangle (9.6,1.2) ; \draw[->] (10.2,0.75) -- (9.8,0.75) ;
  \draw[fill=black!30!white] (8.2,1.2) rectangle (8.9,1.5) ; \draw[fill=black!20!white] (8.9,1.2) rectangle (9,1.5) ;  
  \draw[fill=black!30!white] (7,1.5) rectangle (8.2,1.8) ;
  \draw[fill=black!40!white] (5.7,1.8) rectangle (6.5,2.1) ; \draw [fill=black!30!white](6.5,1.8) rectangle (7,2.1) ;

  \draw[fill=black!40!white] (9.6,0.9) rectangle (10,1.2) ;
  \draw[fill=black!40!white] (9,1.2) rectangle (9.6,1.5) ;  
  \draw[fill=black!50!white] (8.2,1.5) rectangle (8.8,1.8) ; \draw[fill=black!40!white] (8.8,1.5) rectangle (9,1.8) ;
  \draw[fill=black!50!white] (7.1,1.8) rectangle (8.2,2.1) ;
\draw[fill=black!60!white] (7,1.8) rectangle (7.1,2.1) ;

  \draw[fill=black!60!white] (9.6,1.2) rectangle (10,1.5) ;
  \draw[fill=black!60!white] (9,1.5) rectangle (9.6,1.8) ;  
  \draw[fill=black!60!white] (8.6,1.8) rectangle (9,2.1) ;

  \draw[very thick] (9.6,0.3) rectangle (10,1.5) ;
  \draw[very thick] (9,0.6) rectangle (9.6,1.8) ;
  \draw[very thick] (7,1.2) rectangle (8.2,2.1) ;
  \draw[very thick] (5.7,1.5) rectangle (7,2.1) ;  
 \draw[very thick] (4.5,1.8) rectangle (5.7,2.1) ;
\draw[very thick] (8.2,0.9) -- (9,0.9) -- (9,2.1) -- (8.6,2.1) -- (8.6,1.8)-- (8.2,1.8) -- (8.2,0.9) ;

  \draw (0,0.3) rectangle (9.6,0.6) ;
  \draw (0,0.6) rectangle (9,0.9) ;
  \draw (0,0.9) rectangle (8.2,1.2) ;  
\end{tikzpicture}
\caption{We cut the union of the levels $A_{k+1},\dots,A_{C_{0}}$ into small pieces $X_{a,b}$ of size $|A_{a-1}|-|A_{a}|$ for $1\leq a\leq b\leq k$. For $a=1,\dots,k$, we consider the block $X_{a,a}\cup\dots \cup X_{a,k}$ and partition it into $\approx |X_{a,a}|$ chains, whose collection is denoted by $\mathcal{C}_{a}$. Finally, we find a chain decomposition $\mathcal{D}_{0}$ of $A_{0}\cup\dots\cup A_{k}$ into $M$ chains, and attach the chains in $\mathcal{C}_{a}$ to those chains of $\mathcal{D}_{0}$ that end in $A_{a-1}$.}
\label{figure:chains}
\end{figure}
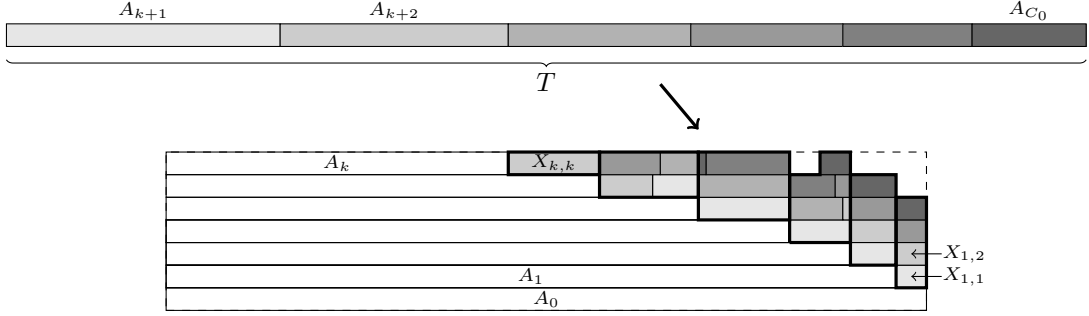

As a reminder, for $l=0,\dots,k-1$, the \emph{$l$-th diagonal} is the union $\bigcup_{a:(a,a+l)\in I}X_{a,a+l}$. Say that a diagonal is \emph{complete} if $(k-l,k)\in I$. Let $\mu\leq k$ be the largest number such that the $(k-\mu)$-th diagonal is not complete. Then for every $1\leq a\leq k$, the number of indices $b$ such that $(a,b)\in I$ is at least $k+1-a-\mu$ (note that this number might be negative). 

Let us estimate $\mu$.
\begin{claim}\label{claim:missingblocks}
 $\mu=O(n^{1/3}).$
\end{claim}
\begin{proof}
If the $l$-th diagonal is complete, then it contains $M-|A_{k-l}|$ elements. Consider the inequality $(k-1)M<|B|$. We have $|B|=\sum_{i=0}^{k}|A_{i}|+|T|+|Z|$, so this inequality can be rewritten as $|T|+|Z|>-2M+\sum_{i=1}^{k}(M-|A_{i}|)$. Since the $(k-\mu)$-th diagonal is not complete, we have $|T|\leq \sum_{l=0}^{k-\mu}(M-|A_{k-l}|)$, which then implies $|Z|\geq -2M+\sum_{i=1}^{\mu-1} (M-|A_{i}|)$. By Claim~\ref{claim:binomial},~(4), we have $|A_{i}|\leq  M\left(1-\frac{i^{2}}{4n}\right)$. 

Therefore, 
$$|Z|\geq -2M+M\sum_{i=1}^{\mu-1}\frac{i^{2}}{4n}\geq \frac{(\mu-1)^{3}M}{12n}-2M.$$
From this, and using that $|Z|\leq n^{-2/3}2^n<M$, we conclude that $\mu=O(n^{1/3}).$
\end{proof}

For $(a,b)\in I$, let $\phi(a,b)$ be the set of indices $r$ such that $A_{r}\cap X_{a,b}\neq\emptyset$. Say that the index $(a,b)\in I$ is \emph{whole} if $|\phi(a,b)|=1$, and say that $(a,b)$ is \emph{shattered} otherwise. In other words, $(a,b)$ is whole if $X_{a,b}$ is completely contained in a level, and shattered otherwise. Clearly, the number of shattered indices in $I$ is at most $C_{0}$ as $X_{a,b}$ is shattered if there exists $r$ such that $X_{a,b}$ contains the last point of $A_{r}$ and the first point of $A_{r+1}$ with respect to $\prec$. 

 The proof of the following claim is rather technical and does not add much to the reader's understanding of the paper, hence we have moved it to the Appendix.
 
\begin{claim}\label{claim:calc}
Let $1\leq a\leq k$ and $a\leq b<b'\leq k$. Then $\phi(a,b)$ and $\phi(a,b')$ are disjoint.
\end{claim}

\emph{Remark.} This claim is quite important for our proof to work, and it seems more of a coincidence that it is actually true, rather than having some combinatorial reason behind it. To prove the claim, we do delicate calculations with binomial coefficients, which the interested reader can find in the Appendix.

For $a=1,\dots,k$, let $$K_{a}=\bigcup_{\substack{b : (a,b)\in I\\ (a,b)\mbox{\footnotesize\ is whole}}} X_{a,b}.$$ Then $K_{a}$ is the union of $|A_{a-1}|-|A_{a}|$ sized random subsets of distinct levels, where the fact that these levels are distinct follows from Claim \ref{claim:calc}. In what comes, we would like to partition $K_{a}$ into roughly $|A_{a-1}|-|A_{a}|$ chains, most of them of size $\approx k-a$. In order to do this, it is enough to show that the size of the largest antichain of $K_{a}$ is not much larger than $|A_{a-1}|-|A_{a}|$. To bound the size of this largest antichain, we use the celebrated container method. The \emph{graph} container method, which we will use in the present work, dates back to works of Kleitman and Winston~\cite{kw1,kw2} from more than 30 years ago; for more recent applications see~\cite{btw,wojteksurvey}.  We will use a multi-stage version of the method, this idea has first appeared in~\cite{BSperner}.

\subsection{Containers}
In this section, we construct a small family  $\mathcal{C}$ of subsets of $T$, which we shall refer to as \emph{containers}, such that every antichain of $T$ is contained in some element of $\mathcal{C}$, and each $C\in\mathcal{C}$ has small mass, where we use the following notion of mass.

If $\mathcal{F}\subset 2^{[n]}$, the \emph{Lubell-mass} of $\mathcal{F}$ is 
$$\ell(\mathcal{F})=\sum_{x\in \mathcal{F}}\frac{1}{\binom{n}{|x|}}.$$

Next, we show that any family of large Lubell-mass must contain an element that is comparable to many other elements. 

\begin{claim}\label{claim:maxdeg}
Let $\delta>0$, $r$ is positive integer and let $\mathcal{F}\subset B$ such that $\ell(\mathcal{F})=r+\delta$. Then there exists $x\in \mathcal{F}$ such that $x$ is comparable with at least $\frac{\delta}{(r+\delta)r!}\cdot(\frac{n}{2})^{r}$ elements of $\mathcal{F}$. 
\end{claim}

\begin{proof}
For each $x\in \mathcal{F}$, consider the number of elements of $\mathcal{F}$ comparable with $x$, and let $\Delta$ be the maximum of these numbers.

Let $C$ be a maximal chain in $2^{[n]}$ chosen randomly from the uniform distribution. Note that $\mathbb{E}(|C\cap\mathcal{F}|)=\ell(\mathcal{F})=r+\delta$. Let $N$ be the number of pairs $(x,y)$ in  $C\cap \mathcal{F}$ such that $x\subset y$ and $|y|-|x|\geq r$. On one hand, we have $N\geq |C\cap \mathcal{F}|-r$, hence $\mathbb{E}(N)\geq \delta$. On the other hand, if $x,y\in \mathcal{F}$ such that $x\subset y$ and $|y|-|x|\geq r$, then 
$$\mathbb{P}(x,y\in C)=\frac{|x|!(|y|-|x|)!(n-|y|)!}{n!}=\frac{1}{\binom{n}{|y|}}\cdot\frac{1}{\binom{|y|}{|x|}}\leq \frac{1}{\binom{n}{|y|}}\cdot\frac{1}{\binom{m+r}{r}}\leq \frac{r!}{\binom{n}{|y|}}\left(\frac{2}{n}\right)^{r},$$
noting that $|y|\geq m+r\geq \frac{n}{2}+r$. For $y\in \mathcal{F}$, let $D(y)=\{x\in \mathcal{F}:x\subset y, |y|-|x|\geq r\}$. Then we can write
\begin{align*}
    \mathbb{E}(N)&=\sum_{y\in\mathcal{F}}\sum_{x\in D(y)} \mathbb{P}(x,y\in C)\leq \sum_{y\in\mathcal{F}}|D(y)|\frac{r!}{\binom{n}{|y|}}\left(\frac{2}{n}\right)^{r}\\&\leq \Delta r!\left(\frac{2}{n}\right)^{r}\ell(\mathcal{F})=\Delta r!\left(\frac{2}{n}\right)^{r}(r+\delta).
\end{align*}

Comparing the right hand side with the lower bound $\delta\leq \mathbb{E}(N)$, we get the desired bound $\Delta\geq \frac{\delta}{(r+\delta)r!}\cdot(\frac{n}{2})^{r}$.
\end{proof}

Now we are ready to establish our container lemma. In the proof we will use the above claim only for $r=1,2$. 

\begin{lemma}\label{lemma:container}
There exists a family $\mathcal{C}$ of subsets of $T$ such that 
\begin{enumerate}
    \item $|\mathcal{C}|\leq 2^{2^{n}n^{-3/2+o(1)}}$,
    \item for every $C\in\mathcal{C}$, we have $\ell(C)\leq 1+n^{-1/3+o(1)}$,
    \item if $I$ is an antichain in $T$, then there exists $C\in\mathcal{C}$ such that $I\subset C$.
\end{enumerate}
\end{lemma}

\begin{proof}
  Let $G$ be the comparability graph of $T$, and let $<$ be an arbitrary total ordering on $T$. Let $I$ be an antichain of $T$. We build a container containing $I$ with the help of the following algorithm. 
  
  \begin{description}
      \item[Step 0] Set $S_{0}:=\emptyset$ and $G_{0}:=G$.
      \item[Step $i$] Let $v_{i}$ be the smallest vertex (with respect to $<$) of $G_{i-1}$  with maximum degree. If $\ell(G_{i-1})\geq 1+n^{-1/2}$, then consider two cases. 
 \begin{itemize}
     \item if $v_{i}\not\in I$, then let $G_{i}:=G_{i-1}\setminus\{v_{i}\}$, $S_{i}:=S_{i-1}$ and proceed to step $i+1$,
     \item if $v_{i}\in I$, then let $S_{i}:=S_{i-1}\cup \{v_{i}\}$ and $G_{i}:=G_{i-1}\setminus (\{v_{i}\}\cup N_{G_{i-1}}(v_{i}))$, and proceed to step $i+1$.
 \end{itemize}
 On the other hand, if $\ell(G_{i-1})<1+n^{-1/2}$, then set $S=S_{i-1}$, $f(S)=V(G_{i-1})$ and terminate the algorithm.
  \end{description}
  
   Call the set $S$ a \emph{fingerprint}. Note that $V(G_{i-1})$ only depends on $S$, so the function $f$ is properly defined on the set of fingerprints. Finally, set $C=S\cup f(S)$, then $C$ contains $I$. Let $\mathcal{C}$ be the family of the sets $C$ for every independent set $I$. 
 
 Now let us estimate the size of $S$. We study our algorithm by dividing the steps into phases depending on $\ell(V(G_{i}))$.
 
 \begin{description}
     \item[Phase -1] This phase consists of those steps $i$ for which $\ell(V(G_{i}))\geq 3$, and let $i'$ be the last step in this phase. In every such step, the maximum degree of $V(G_{i})$ is at least $\frac{n^{2}}{24}$ by Claim~\ref{claim:maxdeg} (with $r=2$ and $\delta=1$). If we added $v_{i}$ to $S_{i-1}$, then we have $|V(G_{i})|\leq |V(G_{i-1})|-\frac{n^{2}}{24}$, which means that $|V(G_{i'})|\leq 2^{n}-\frac{|S_{i'}|n^{2}}{24}$. Therefore, $|S_{i'}|\leq \frac{24\cdot 2^{n}}{n^{2}}.$ Let $T_{-1}=S_{i'}$.
     
     \item[Phase 0] This phase consists of those steps $i$ for which $ 3>\ell(V(G_{i-1}))\geq 2$, and let $i_{0}$ be the last step of this phase. Also, let $T_{0}=S_{i_{0}}\setminus T_{-1}$, the set of elements we added to $S$ during this phase. In this phase, we have $|V(G_{i-1})|\leq 3M$ and by 
     Claim~\ref{claim:maxdeg} (with $r=1$ and $\delta=1$), the maximum degree of $V(G_{i-1})$ is at least $\frac{n}{4}$. If $v_{i}\in I$, then we have $|V(G_{i})|\leq |V(G_{i-1})|-\frac{n}{4}$, which means that $|V(G_{i_{0}})|\leq 3M-|T_{0}|\frac{n}{4}$. Therefore, $|T_{0}|\leq \frac{12M}{n}<12\cdot 2^{n}n^{-3/2}.$
     
     \item[Phase r]  For $r=1,\dots,\frac{1}{2}\log_{2} n$, phase $r$ consists of those steps $i$ for which $1+\frac{1}{2^{r-1}}>\ell(V(G_{i-1}))\geq 1+\frac{1}{2^{r}}$. Let $i_{r}$ be the last step of phase $r$ and let $T_{r}=S_{i_{r}}\setminus S_{i_{r-1}}$, the set of elements we added to $S$ during phase $r$. By Claim~\ref{claim:maxdeg} (with $r=1$ and $\delta=\frac{1}{2^{r}}$), the maximum degree of $V(G_{i-1})$ is at least $\frac{n}{2^{r+2}}$. Also, $\ell(V(G_{i_{r-1}})\setminus V(G_{i_{r}}))\leq \frac{1}{2^{r}}$, so $|V(G_{i_{r-1}})\setminus V(G_{i_{r}})|\leq \frac{M}{2^{r}}$. Moreover, $|V(G_{i_{r}})|\leq |V(G_{i_{r-1}})|-|T_{r}|\frac{n}{2^{r+2}}$, which gives
 $$|T_{r}|\leq \frac{4M}{n}\leq \frac{4\cdot 2^{n}}{n^{3/2}}.$$
 \end{description}
 Therefore, in the end of the process, we get $$|S|=\sum_{r=-1}^{\frac{1}{2}\log_{2} n} |T_{r}|\leq \frac{3\cdot 2^{n}\log_{2} n}{n^{3/2}}.$$ Hence, there are at most $$\binom{2^{n}}{3\cdot 2^{n}n^{-3/2}\log_{2} n}=2^{2^{n}n^{-3/2+o(1)}}$$ fingerprints, which is also an upper bound for $|\mathcal{C}|$.  It only remains to bound $\ell(C)$. Recall that $T$ contains only sets of size at most $m+C_0$, $\binom{n}{m+C_{0}}=(1+o(1))n^{-2/3} M$ and $M \leq O(2^n/\sqrt{n})$. Thus we have 
 $$\ell(C)<\ell(f(S))+\ell(S)\leq 1+n^{-1/2}+\frac{|S|}{\binom{n}{m+C_{0}}}\leq 1+n^{-1/2}+(1+o(1))n^{2/3}\frac{|S|}{M}.$$ 
 Here, $n^{2/3}\frac{|S|}{M}=O(n^{-1/3}\log n)$, so $\ell(C)\leq 1+O(n^{-1/3}\log n)$.
  
\end{proof}

\subsection{Antichains}

The aim of this section is to bound the size of the maximal antichain in $K_{a}$. Recall that for $(a,b)\in I$,  $\phi(a,b)$ is the set of indices $r$ such that $A_{r}\cap X_{a,b}\neq\emptyset$, and  
$$K_{a}=\bigcup_{\substack{b: (a,b)\in I\\(a,b)\mbox{\footnotesize\ is whole}}} X_{a,b}.$$

\begin{lemma}\label{lemma:antichain}
 Let $a\geq n^{1/10}$. With probability at least $1-2^{-n^2}$, the size of the maximal antichain of $K_{a}$ is $$\left(1+\frac{n^{o(1)}}{\sqrt{a}}\right)(|A_{a-1}|-|A_{a}|).$$
\end{lemma}

\begin{proof}
 Let $A=|A_{a-1}|-|A_{a}|$, then by Claim~\ref{claim:binomial} (5), we have $A=\Theta(a2^{n}n^{-3/2})$. Let $E$ be the set of indices $b$ such that $(a,b)\in I$ and $(a,b)$ is whole. If $b\in E$, let $r_b$ be the unique index such that $X_{a,b}\subset A_{r_{b}}$. Then $X_{a,b}$ is an $A$ element subset of $A_{r_b}$, chosen from the uniform distribution on all $A$ element subsets. Also, as the sets $\phi(a,b)$ for $b\in E$ are pairwise disjoint by Claim \ref{claim:calc}, the system of random variables $\{X_{a,b}:b\in E\}$ is independent. 

Instead of $X_{a,b}$, it is more convenient to work with the set $Y_{a,b}$ which we get by selecting each element of $A_{r_b}$ independently with probability $p_b=\frac{A}{|A_{r_b}|}$. Indeed, $X_{a,b}=Y_{a,b}|(|Y_{a,b}|=A)$, and $$\mathbb{P}(|Y_{a,b}|=A)=p_b^A(1-p_b)^{|A_{r_b}|-A}\binom{|A_{r_b}|}{A}>\frac{1}{|A_{r_b}|}>2^{-n},$$where the second to last inequality can be seen by observing that the function $f(x)=\mathbb{P}(|Y_{a,b}| = x)$ is increasing for $x\leq |A|$ and decreasing for $x\geq |A|$.

Let $D=\bigcup_{b\in E}Y_{a,b}$ and $U=\bigcup_{b\in E}A_{r_b}$. Let $\mathcal{C}$ be the family of containers of $T$ given by Lemma~\ref{lemma:container}. Let $\delta$ be a real number such that $n^{-1/3+1/20}<\delta<1$, let $C\in\mathcal{C}$ and consider the probability that $W=|C\cap D|$ is larger than $A(1+\delta)$. First of all, we have $$\mathbb{E}(W)=\sum_{b\in E}\frac{A|C\cap A_{r_b}|}{|A_{r_b}|}=A\sum_{b\in E}\frac{|C\cap A_{r_{b}}|}{\binom{n}{m+r_{b}}}=A\ell(|C\cap U|)\leq A(1+n^{-1/3+o(1)}).$$

Now let us estimate the probability that $W\geq (1+\delta)A$. Let $\delta'=(1+\delta)\frac{A}{\mathbb{E}(W)}-1$, then $(1+\delta)A=(1+\delta')\mathbb{E}(W)$. Using the property that $\delta>n^{-1/3+1/20}$, we have $\delta'\geq \delta \frac{A}{2\mathbb{E}(W)}$. But $W$ is the sum of Bernoulli random variables, so we can apply Chernoff's inequality (Claim~\ref{claim:chernoff}). Consider two cases: if $\delta'\leq 1$, then
$$\mathbb{P}(W\geq (1+\delta')\mathbb{E}(W))\leq e^{-\frac{(\delta')^2}{3}\mathbb{E}(W)}\leq e^{-\delta^2 \frac{A^2}{12\mathbb{E}(W)}}\leq e^{-\frac{\delta^2 A}{24}},$$
and if 
$\delta'>1$, then 
$$\mathbb{P}(W\geq (1+\delta')\mathbb{E}(W))\leq e^{-\frac{\delta'}{3}\mathbb{E}(W)}\leq e^{-\delta \frac{A}{6}}< e^{-\frac{\delta^2 A}{24}}.$$
Choose $\delta$ such that $e^{-\frac{\delta^2 A}{24}}|\mathcal{C}|=2^{-2n^{2}}$. Since $A=\Theta(a2^{n}n^{-3/2})$ and $|\mathcal{C}|\leq 2^{2^{n}n^{-3/2+o(1)}}$, we have  $\delta=\frac{n^{o(1)}}{\sqrt{a}}$. Note that $n^{-1/3+1/20}<\delta<1$ holds, so the previous calculations are valid for this choice of $\delta$. By the union bound, the probability that there exists $C\in \mathcal{C}$ such that $|C\cap D|\geq (1+\delta)A$ is at most $|\mathcal{C}| e^{-\frac{\delta^2 A}{24}}=2^{-2n^{2}}.$ But every independent set of $U$ is contained in some $C\in\mathcal{C}$, so the probability $q'$ such that $D$ has no independent set of size larger than $(1+\delta)A$ is at most $2^{-2n^{2}}$. 

 Finally, let $q$ be the probability that $K_{a}$ has an independent set larger than $(1+\delta)A$. Then $q$ is equal to the probability that $D$ has an independent set of size $(1+\delta)A$, conditioned on the event that $|Y_{a,b}|=A$ for $b\in E$. But the probability of this event is at least $2^{-n|E|}> 2^{-n^{2}}$, since $|E|\leq k=O(\sqrt{n})$, so $q\leq q'2^{n^{2}}\leq 2^{-n^{2}}$.
\end{proof}

\subsection{Matchings}
By the previous lemma and by Dilworth's theorem \cite{D50}, we know that $K_{a}$ can be partitioned into slightly more than $(|A_{a-1}|-|A_{a}|)$ chains. We would like to attach most of these chains to a chain decomposition of the union of the levels $A_{0}\cup\dots\cup A_{k-1}$. This section is devoted to the following lemma, which deals with this problem. For $a=1,\dots,k$, let $B_a$ be the bipartite graph with vertex classes $A_{a-1}$ and $A_{a}\cup X_{a,a}$, where the edges between the two vertex classes are the comparable pairs. Note that $B_{a}$ is a \emph{balanced} bipartite graph, that is, $|A_{a-1}|=|A_{a}\cup X_{a,a}|$.

\begin{lemma}\label{lemma:matching}
 If $(a,a)$ is whole, then with probability at least $1-2^{-n}$, there exists a matching $M_a$ in $B_a$ such that $M_a$ covers every element of $A_a$, and $M_{a}$ covers all but $O(\frac{2^{n}}{n^{5/4}})$ elements of $X_{a,a}$.
\end{lemma}

We prepare the proof of this lemma with a number of simple claims, the first one of which is a form of the LYM inequality. 

\begin{claim}\label{claim:normalizedmatching}
Let $i,j\in \{0,\dots,n-m\}$, $i\neq j$. Let $G$ be the bipartite graph with vertex classes $A_i$ and $A_j$ such that the edges of $G$ are the comparable pairs. Then for every $X\subset A_i$, we have $$\frac{|X|}{|A_{i}|}\leq \frac{|N_{G}(X)|}{|A_j|}.$$ 
\end{claim}

\begin{proof}
  Suppose that $i<j$, the other case can be handled in a similar manner. Let $e$ denote the number of edges between $X$ and $N_{G}(X)$. Counting $e$ from the vertices in $X$, we get $e=|X|\binom{n-m-i}{j-i}$. Counting the edges by the vertices in $N_{G}(X)$, we get $e\leq |N_{G}(X)|\binom{j+m}{j-i}$. Therefore, $|X|\binom{n-m-i}{j-i}\leq |N_{G}(X)|\binom{j+m}{j-i}$, which is equivalent to $\frac{|X|}{|A_{i}|}\leq \frac{|N_{G}(X)|}{|A_j|}.$
\end{proof}

\begin{claim}\label{claim:hall}(Defect version of Hall's theorem~\cite{H35}, see also~\cite{ADH98})
 Let $G$ be a bipartite graph with vertex classes $A$ and $B$, and let $\Delta$ be a positive integer. Suppose that for every $X\subset A$, we have $|N_{G}(X)|\geq |X|-\Delta$. Then $G$ contains a matching of size at least $|A|-\Delta$.
\end{claim}

\begin{claim}\label{claim:levels}
 Let $0\leq i<n-m$ and let $G$ be the bipartite graph with vertex classes $A_{i}$ and $A_{i+1}$ in which the edges are the comparable pairs. Then there exists a complete matching from $A_{i+1}$ to $A_{i}$ for $i=0,\dots,n-m$.
\end{claim}

\begin{proof}
  This follows easily from Claim~\ref{claim:normalizedmatching} and Hall's theorem (Claim~\ref{claim:hall} with $\Delta=0$). Indeed, for every $X\subset A_{i}$, we have $|N_{G}(X)|\geq \frac{|A_{i}|}{|A_{i+1}|}|X|\geq |X|$, so Hall's condition is satisfied. Therefore, there exists a matching of size $|A_{i+1}|$. 
\end{proof}

\begin{corollary}\label{cor:leftover}
Let $T'=[n]^{(\geq m+k)}$. Then $T'$ can be partitioned into $|A_{k}|$ chains.
\end{corollary}

\begin{proof}
For $i=k,\dots,n-m-1$, let $M_{i}$ be a complete matching from $A_{i+1}$ to $A_{i}$. For $x\in A_{k}$, let $C_{x}$ be the chain with elements $x=x_{0}\subset\dots\subset x_{l}$, where $x_{j}$ is matched to $x_{j+1}$ in $M_{k+j}$ for $j=0,\dots,l-1$, and $x_{l}$ is not covered by the matching $M_{k+l}$. Then $\{C_{x}\}_{x\in A_{k}}$ is a chain partition of $T'$ into $|A_{k}|$ chains.  
\end{proof}

\begin{claim}\label{claim:maxmatching}(see for example~\cite{ADH98})
 Let $G$ be a bipartite graph and $M$ be a matching in $G$. Then there exists a maximal sized matching $M'$ in $G$ such that $V(M)\subset V(M')$.
\end{claim}

\begin{claim}\label{claim:k+1-k+2}
  $X_{a,a}\subset (A_{k+1}\cup A_{k+2})$, and in particular, $X_{k-1,k-1},X_{k,k}\subset A_{k+2}$.
\end{claim}
\begin{proof}
By numerical calculations, we have $0.4M< |A_{k+2}|<|A_{k+1}|< |A_{k}|<0.46 M$, so the inequalities $|A_{k+1}|<M-|A_{k}|<|A_{k+1}|+|A_{k+2}|$ hold. Here, $M-|A_{k}|$ is the size of the first diagonal, which contains $X_{1,1},\dots,X_{k,k}$. The inequalities show that this diagonal contains $A_{k+1}$ and a constant proportion of $A_{k+2}$. But then as $X_{k-1,k-1},X_{k,k}$ are the last elements of this diagonal, we have $X_{k-1,k-1},X_{k,k}\subset A_{k+2}$.
\end{proof}

Now we are ready to prove the main lemma of this section.

\begin{proof}[Proof of Lemma~\ref{lemma:matching}]
  As $(a,a)$ is whole, there exists an index $r$ such that $X_{a,a}\subset A_{r}$, and let $A=|A_{a-1}|-|A_a|$. By Claim \ref{claim:k+1-k+2}, we have $r\in \{k+1,k+2\}$, and in particular, $r=k+2$ if $a\in \{k-1,k\}$.  Similarly as before, instead of working with the random set $X_{a,a}$, we will work with the set $Y$ we get by selecting each element of $A_r$ independently with probability $p=\frac{A}{|A_r|}$. Indeed, $X_{a,a}$ has the same distribution as $Y|(|Y|=A)$, and $\mathbb{P}(|Y|=A)\geq \frac{1}{|A|}\geq 2^{-n}$. Let $E$ be the bipartite graph with vertex classes $A_{a-1}$ and $A_{a}\cup Y$, where the edges are given by the comparable pairs, and let $E'$ be the subgraph of $E$ induced on $A_{a-1}\cup Y$.

Consider the degrees of $E'$ in $A_{a-1}$. Every $x\in A_{a-1}$ is comparable with exactly $d=\binom{n-m-a+1}{r-(a-1)}$ elements of $A_r$, so for every $x\in A_{a-1}$, $\mathbb{E}(\deg_{E'}(x))=pd$.

Next, let us bound $pd$. Consider two cases. 
\begin{description}
\item[$k=a$] We have $r-(a-1)=3$, so 
$$d=\binom{n-m-a+1}{r-(a-1)}\geq \binom{n/3}{3}=\Omega(n^3).$$ 
Also,  we have $p=\frac{A}{|A_r|}=\Omega(n^{-1/2})$ by the following estimates: $A=\Theta(a2^{n}n^{-3/2})=\Omega(2^{n}n^{-1})$ by Claim \ref{claim:binomial}, (5), and $|A_{r}|=\Theta(2^{n}n^{-1/2})$ by  Claim \ref{claim:binomial}, (1)-(2). Therefore, $pd=\Omega(n^{5/2})$.

\item[$a\leq k$] We have $r-(a-1)\geq 4$. Indeed, if $a=k-1$, then $r=k+2$, and if $a\leq k-2$, then $r\geq k+1$. But then $$d=\binom{n-m-a+1}{r-(a-1)}\geq \binom{n/3}{4}=\Omega(n^4).$$ 
Also $p=\frac{A}{|A_r|}=\Omega(n^{-1})$ as $A=\Theta(a2^{n}n^{-3/2})=\Theta(2^{n}n^{-3/2})$ by Claim \ref{claim:binomial}, (5), and $|A_{r}|=\Theta(2^{n}n^{-1/2})$ by  Claim \ref{claim:binomial}, (1)-(2). Therefore, we get $pd=\Omega(n^{3})=\Omega(n^{5/2})$.
\end{description}

Now consider the degree of $x$ in $E'$. As $\deg_{E'}(x)$ is the sum of independent Bernoulli random variables, we can apply Chernoff's inequality with $\delta<1$ (Claim~\ref{claim:chernoff}) to get
$$\mathbb{P}(\deg_{E'}(x)\geq (1+\delta)pd)\leq e^{-\frac{\delta^{2}pd}{3}}.$$
Choose $\delta$ such that $\delta^{2}pd=12n$, then  $\delta=O(n^{-3/4})$. Let $\mathcal{E}$ be the event that there exists $x\in A_{a-1}$ such that $\deg_{E'}(x)\geq (1+\delta)pd$. By the union bound, $\mathbb{P}(\mathcal{E})\leq |A_{a-1}|e^{-4n}<2^{-2n}.$ Moreover, $\mathbb{P}(\mathcal{E}|(|Y|=A))\leq \frac{\mathbb{P}(\mathcal{E})}{\mathbb{P}(|Y|=A)}\leq 2^{-n}$.

To finish the proof, it is enough to show that if $\overline{\mathcal{E}}\cap (|Y|=A)$ happens, then the desired matching exists. Let $d'=\binom{m+r}{r-(a-1)}$, then the degree of every vertex in $Y$ is $d'$. Let $U\subset Y$, $V=N_{E}(U)$ and let $f$ be the number of edges between $U$ and $V$. We have 
$$d'|U|=f\leq (1+\delta)pd|V|,$$
which implies $|V|\geq \frac{d'}{d(1+\delta)p}|U|.$ Note that $\frac{d'}{pd}=\frac{1}{p}\cdot \frac{|A_{a-1}|}{|A_{r}|}=\frac{|A_{a-1}|}{A}$ and $\frac{|U|}{|A|}\leq 1$, so
$$|V|\geq |U|\frac{|A_{a-1}|}{A(1+\delta)}\geq|U|\frac{|A_{a-1}|}{A}(1-\delta)\geq |U|\frac{|A_{a-1}|}{A}-\delta|A_{a-1}|.$$
Also, by Claim~\ref{claim:normalizedmatching}, for every $U'\subset A_{a}$, we have $|N_{E}(U')|\geq |U'|\frac{|A_{a-1}|}{|A_{a}|}$.

Now we show that Hall's condition holds in $E$ with defect $\Delta=\delta|A_{a-1}|=O(\frac{2^{n}}{n^{5/4}})$, that is, for every $U_{0}\subset A_{a}\cup Y$, we have $|N_{E}(U_{0})|\geq |U_{0}|-\Delta$. If this is true, then Claim~\ref{claim:hall} implies that there exists a matching of size at least $(1-\delta)|A_{a-1}|$ in $E$. But there exists a complete matching $M$ in $E$ from $A_{a}$ to $A_{a-1}$ by Claim~\ref{claim:levels}, so  there exists a matching $M'$ of maximal size that covers every element of $A_{a}$ by Claim~\ref{claim:maxmatching}. Then $M'$ satisfies the desired properties.

Let $U_{0}\subset A_{a}\cup Y$, $U=U_{0}\cap Y$ and $U'=U_{0}\cap Y$. Then 
\begin{align*}
|N_{E}(U_{0})|&\geq \max\{|N_{E}(U)|,|N_{E}(U')|\}\\
&\geq |A_{a-1}|\max\left\{\frac{|U|}{A}-\delta,\frac{|U'|}{|A_{a}|}\right\}\\
&\geq |A_{a-1}|\max\left\{\frac{|U|}{A},\frac{|U'|}{|A_{a}|}\right\}-\Delta.
\end{align*}
Let $\alpha=\frac{|U|}{A}$ and $\beta=\frac{|U'|}{|A_{a}|}$. If $\alpha\geq \beta$, then $|U'|\leq \alpha |A_{a}|$ and $|U_{0}|\leq \alpha(|A_{a}|+A)=\alpha |A_{a-1}|$. Therefore, $|N_{E}(U_{0})|\geq |A_{a-1}|\alpha-\Delta \geq |U_{0}|-\Delta$. We can proceed similarly if $\alpha<\beta$. This finishes the proof.
\end{proof}

\subsection{The proof of Theorem \ref{thm:mainthm}}

For $x\in 2^{[n]}$, let $x^c=[n]\setminus x$, and for $F\subset 2^{[n]}$, let $$\overline{F}=\{x^c: x\in F\}.$$

Fix $\lambda=n^{-1/16}$. It is enough to partition $B$ into $\binom{n}{\lfloor n/2\rfloor}$ chains such that all but at most $O(Mn^{-1/8})$ of the chains have size between $k-3\lambda k$ and $k+3\lambda k$. Indeed, let $\mathcal{D}$ be such a chain partition, and for $x\in A_{0}$, let $D_{x}\in \mathcal{D}$ be the chain containing $x$. If $n$ is even, let $D_{x}^+=D_{x}\cup \overline{D_{x^c}}$. If $n$ is odd, let $\tau: A_{0}\rightarrow [n]^{(\frac{n-1}{2})}$ be an arbitrary bijection such that $\tau(x)\subset x$ for every $x\in A_0$, and set $D^+_{x}=D_x\cup \overline{D}_{\tau(x)^c}$. Then $\mathcal{D}^+=\{D_x^+:x\in A_{0}\}$ is a chain partition of $2^{[n]}$ into $\binom{n}{\lfloor n/2\rfloor}$ chains with the desired properties.

In the rest of this section, we prove that there exists a chain partition of $B$ with the properties above.

By Lemma~\ref{lemma:antichain} and Lemma~\ref{lemma:matching}, there is a choice for the sets $X_{a,b}$, $(a,b)\in I$ such that for $n^{1/10}<a\leq k$, the size of the maximal antichain in $K_{a}$ is $(1+\frac{n^{o(1)}}{\sqrt{a}})(|A_{a-1}|-|A_{a}|)$, and there is a matching $M_{a}$ in $B_{a}$ that covers every element of $A_{a}$, and covers all but at most $O(2^{n}n^{-5/4})$ elements of $X_{a,a}$.

First, we shall cover most elements of $B$ by chains, most of whose size is between $k(1-3\lambda)$ and $k+1$, while collecting certain elements of $B$ which are not covered into a set $\mathcal{L}$. We refer to the elements of $\mathcal{L}$ as \emph{leftovers}. First of all, put every element $x\in B$ satisfying $|x|\geq m+C_{0}$ into $\mathcal{L}$. Then we added at most $2^{n}n^{-2/3}$ elements to $\mathcal{L}$. Also, we put every element of $X_{a,b}$ for $a\leq n^{1/10}$ and $a<b\leq k$ in $\mathcal{L}$. Then, by Claim \ref{claim:binomial}, (5), we put at most $$\sum_{a\leq n^{1/10}}\sum_{b=a}^{k} |X_{a,b}|\leq k\sum_{a\leq n^{1/10}} (|A_{a-1}|-|A_{a}|)=O(2^{n}n^{-4/5})$$ elements in $\mathcal{L}$. So far $$|\mathcal{L}|=O(2^{n}n^{-2/3}).$$

For $a=1,\dots,k$, say that $a$ is shattered if the number of indices $b$ such that $(a,b)\in I$ and $(a,b)$ is shattered is at least $\lambda k$. If $a$ is shattered, then put every element of $\bigcup_{b:(a,b)\in I} X_{a,b}$ into $\mathcal{L}$. In total, there are less than $C_{0}$ shattered pairs $(a,b)$, so the number of shattered indices $a$ is at most $\frac{C_{0}}{\lambda k}$. The size of the set $\bigcup_{b:(a,b)\in I} X_{a,b}$ is $$\sum_{b: (a,b)\in I}|X_{a,b}|=O(ka2^{n}n^{-3/2})=O(2^{n}n^{-1/2}),$$ so we added at most $O(\frac{C_{0}}{\lambda k}2^{n}n^{-1/2})=2^{n}n^{-7/16+o(1)}$ elements to $\mathcal{L}$.

Also, for every $(a,b)\in I$, if $(a,b)$ is shattered, put every element of $X_{a,b}$ into $\mathcal{L}$. The number of shattered sets is less than $C_{0}$, and $|X_{a,b}|=O(a2^{n}n^{-3/2})=O(2^{n}n^{-1})$, so we added at most $O(C_{0}2^{n}n^{-1})=2^{n}n^{-1/2+o(1)}$ elements to $\mathcal{L}$.  So far $$|\mathcal{L}|\leq 2^{n}n^{-7/16+o(1)}.$$

Now let $n^{1/10}<a\leq k-1$ be such that $a$ is not shattered. Let $A=|A_{a-1}|-|A_{a}|=\Theta(a2^{n}n^{-3/2})$, and let $r$ be the size of the set $\{b:(a,b)\in I,|\phi(a,b)|=1\}$. Then $k+1-a\geq r\geq k+1-a-\lambda k-\mu>k+1-a-2\lambda k$, where  $\mu=O(n^{1/3})<\lambda k$ by Claim~\ref{claim:missingblocks}. Also, we have $$|K_{a}|=rA.$$ By the well known theorem of Dilworth~\cite{D50}, $K_{a}$ can be partitioned into at most $(1+\frac{n^{o(1)}}{\sqrt{a}})A$ chains, let $\mathcal{C}_{a}$ denote the collection of chains in such a chain decomposition. Say that a chain $L\in \mathcal{C}_{a}$ is short, if $|L|\leq r-\lambda k$, and let $N_{\mbox{\tiny short}}$ denote the number of short chains. The size of every chain in $K_{a}$ is at most $r$, so 
$$rA=|K_{a}|\leq (r-\lambda k)N_{\mbox{\tiny short}}+(|\mathcal{C}_{a}|-N_{\mbox{\tiny short}})r,$$
which implies 
$$N_{\mbox{\tiny short}}<\frac{n^{o(1)}r}{\lambda k\sqrt{a}}A\leq \frac{\sqrt{a}}{\lambda}2^{n}n^{-3/2+o(1)}\leq 2^{n}n^{-19/16+o(1)}.$$
Say that a chain in $\mathcal{C}_{a}$ is \emph{irrelevant}, if its minimum is not in $X_{a,a}$. Then the number of irrelevant chains is 
$$N_{\mbox{\tiny irr}}=|\mathcal{C}_{a}|-|X_{a,a}|<\frac{n^{o(1)}}{\sqrt{a}}A=\sqrt{a}2^{n}n^{-3/2+o(1)}=O(2^{n}n^{-5/4}).$$ 
Finally, say that a chain $L\in \mathcal{C}_{a}$ is \emph{sad}, if its minimum $z$ is in $X_{a,a}$, but $z$ is not covered by the matching $M_{a}$. Then the number of sad chains is 
$$N_{\mbox{\tiny sad}}=O(2^{n}n^{-5/4}).$$ 
Let $\mathcal{C}_{a}^{*}$ be the set of chains in $\mathcal{C}_{a}$ that are neither short, irrelevant, nor sad, and let $L_{a}\subset K_{a}$ be the set of elements that are not covered by any chain in $\mathcal{C}_{a}^{*}$. Then
 $$|L_{a}|\leq k(N_{\mbox{\tiny short}}+N_{\mbox{\tiny irr}}+N_{\mbox{\tiny sad}})\leq  kn^{-19/16+o(1)}\leq 2^{n}n^{-11/16+o(1)}.$$ 
Add every element of $L_{a}$ for $n^{1/10}<a\leq k$ to the set of leftovers $\mathcal{L}$. In total, we added at most $k2^{n}n^{-11/16}=2^{n}n^{-3/16+o(1)}$ elements to $\mathcal{L}$. At this point, we have $|\mathcal{L}|=2^{n}n^{-3/16+o(1)}$, and we do not add any more elements to $\mathcal{L}$.

Construct the family of chains $\mathcal{D}$ as follows. First, using the matchings $M_{1},\dots,M_{k}$, construct a chain decomposition $\mathcal{D}_{0}$ of $\bigcup_{i=0}^{k}A_{k}$. For $x\in A_{0}$, let $D_{x}$ be the chain $x=x_{0}\subset...\subset x_{l}$, where $x_{i-1}$ is matched to $x_{i}$ in $M_{i}$ for $i=1,\dots,l$, and either $l=k$, or $x_{l}$ is not matched to any element of $A_{l+1}$ in $M_{l+1}$. Then $\mathcal{D}_{0}=\{D_{x}:x\in A_{0}\}$ is a chain decomposition of $\bigcup_{i=0}^{k}A_{k}$ into $M$ chains such that if a chain has maximum element in $A_{l}$, then the size of the chain is exactly $l+1$. 

Now consider some $D\in \mathcal{D}_{0}$. If $y\in A_{a-1}$ is the maximum element of $D$, $y$ is matched to some $z\in X_{a,a}$ in $M_{a}$, and there exists $C\in\mathcal{C}^{*}_{a}$ such that $z$ is the minimal element of $C$, then let $D^{+}=D\cup C$ and say that $D$ is \emph{compatible}. Noting that $|D|=a$ and $k+1-a\geq |C|\geq k+1-a-3\lambda k$, we have $k+1\geq |D^{+}|\geq k+1-3\lambda k$. Also, if $a=k+1$, then set $D^{+}=D$ and say that $D$ is also compatible. In this case, $|D|=k+1$. Otherwise, if $a\leq k$, and $y$ is not matched to some $z\in X_{a,a}$, or $z$ is not the minimal element of a chain in $\mathcal{C}^{*}_{a}$, then let $D^{+}=D$, and say that $D$ is incompatible. Set $\mathcal{D}=\{D^{+}:D\in\mathcal{D}_{0}\}$. The number of incompatible chains with maximum element in $A_{a-1}$ is at most the number of short and sad chains in $\mathcal{C}_{a}$, which is at most $$N_{\mbox{\tiny short}}+N_{\mbox{\tiny sad}}\leq 2^{n}n^{-19/16+o(1)}.$$ 
 Therefore, the total number of incompatible chains across every $a$ is at most $k2^{n}n^{-19/16+o(1)}=2^{n}n^{-11/16+o(1)}.$

To summarize our progress so far, we constructed a family $\mathcal{D}$ of $M$ chains such that $\mathcal{D}$ partitions $B\setminus\mathcal{L}$, and all but at most $2^{n}n^{-11/16+o(1)}$ chains in $\mathcal{D}$ have size between $k+1-3\lambda k$ and $k+1$.

It only remains to partition $\mathcal{L}$ into a few chains such that each of these chains can be attached to an element of $\mathcal{D}$. This guarantees that the number of chains remains $M$ and only a few of the chains get longer. Let $\mathcal{S}$ be a family of $|A_k|$ chains that partition $[n]^{(\geq m+k)}$, see Corollary \ref{cor:leftover}. Then $\mathcal{S}'=\{S\cap \mathcal{L}: S\in\mathcal{S}\}$ forms a chain partition of the leftover elements. We form our final chain partition by gluing the chains of $\mathcal{S}'$ to certain chains of $\mathcal{D}$. For $x\in A_k$, let $S_x$ be the unique chain containing $x$. For $D\in\mathcal{D}$, let $D^{*}=D\cup (S_x\cap \mathcal{L})$ if the maximum element of $D$ is in $A_k$, and this maximum element is $x$. Otherwise, let $D^{*}=D$. Then $\mathcal{D}^{*}=\{D^{*}:D\in\mathcal{D}\}$ is a chain partition of $B$ into $M$ chains. We show that $\mathcal{D}$ satisfies the desired properties.

Let us count the number of chains $D\in\mathcal{D}$ such that either $|D^{*}|\leq k+1-3\lambda k$, or $|D^*|\geq k+1+\lambda k$. If $|D^{*}|\leq k+1-3\lambda k$, then $D^{*}$ is an incompatible chain in $\mathcal{D}_{0}$, so the number of such chains is at most $2^{n}n^{-11/16+o(1)}=Mn^{-3/16+o(1)}$. On the other hand, if $|D^*|\geq k+1+\lambda k$, then $|D|=k+1$ and there exists $x\in A_k$ such that $D^*=D\cup (S_x \cap \mathcal{L})$. But then $|S_x\cap \mathcal{L}|\geq \lambda k$, so the number of such chains is at most $\frac{|\mathcal{L}|}{\lambda k}=2^{n}n^{-5/8+o(1)}=M n^{-1/8+o(1)}$.\hfill$\Box$

\begin{proof}[Proof of Corollary \ref{remark}]
Let $\mathcal{C}_{0}$ be the family of chains $C\in \mathcal{C}$ such that $||C|-s|\geq n^{\frac{1}{2}-\frac{1}{20}}$. 
By Theorem \ref{thm:mainthm}, $|\mathcal{C}_{0}| \leq M n^{-1/8+o(1)}$ 
Also, let $\mathcal{C}_{1}\subset \mathcal{C}_{0}$ be the family of chains $C$ such that $|C|\geq \sqrt{n}\log n$, and let $\mathcal{C}_{2}=\mathcal{C}_{0}\setminus\mathcal{C}_{1}$.

First, note that every chain of size $\sqrt{n}\log n$ must contain a set of size either at least $\frac{n+\sqrt{n}\log n}{2}$, or at most $\frac{n-\sqrt{n}\log n}{2}$. But by Claim~\ref{claim:binomial}, (3), we have $$\left|[n]^{(\leq \frac{n-\sqrt{n}\log n}{2})}\right|=\left|[n]^{(\geq \frac{n+\sqrt{n}\log n}{2})}\right|\leq 2^{n}e^{-(\log n)^{2}/2},$$ so $|\mathcal{C}_{1}|\leq 2^{n+1}e^{-(\log n)^{2}/2}$. Therefore,
$$\sum_{C\in\mathcal{C}_{1}}|C|\leq 2^{n+1}e^{-(\log n)^{2}/2}n=O\left(\frac{2^{n}}{n}\right).$$
Second, since $|\mathcal{C}_{2}|<Mn^{-\frac{1}{8}+o(1)}$, we can write $$\sum_{C\in\mathcal{C}_{2}}|C|<Mn^{-\frac{1}{8}+o(1)}\sqrt{n}\log n=2^{n}n^{-\frac{1}{8}+o(1)}.$$
Thus, $\sum_{C\in\mathcal{C}_{0}}|C|\leq 2^{n}n^{-\frac{1}{8}+o(1)}.$
\end{proof}

\section{Applications}\label{sect:applications}

\subsection{Minimal Sperner graphs--Proof of Theorem~\ref{thm:sperner}}\label{sect:sperner}

Let $M=\binom{n}{\lfloor n/2\rfloor}$. The lower bound follows from Tur\'an's theorem~\cite{T41}. Indeed, for any graph $G$, if $\alpha(G)$ denotes the independence number of $G$, then $|E(G)|\geq \frac{|V(G)|^{2}}{2\alpha(G)}-\frac{|V(G)|}{2}.$ Plugging $|V(G)|=2^{n}$ and $\alpha(G)=M$ into this formula, we get $$|E(G)|\geq \frac{2^{2n}}{2M}-\frac{2^{n}}{2}=\left(\sqrt{\frac{\pi}{8}}+o(1)\right)2^{n}\sqrt{n}.$$

It only remains to prove the upper bound. Let $s=\frac{2^{n}}{M}=(\sqrt{\frac{\pi}{2}}+o(1))\sqrt{n}$. Let $\mathcal{C}$ be a family of $M$ chains partitioning $2^{[n]}$ such that all but at most $n^{-\frac{1}{8}+o(1)}$ proportion of the chains in $\mathcal{C}$ have size  $(1+O(n^{-1/16}))s$. Such a chain decomposition exists by Theorem~\ref{thm:mainthm}. Let $G$ be the graph on $2^{[n]}$ in which $x$ and $y$ are joined by an edge if $x$ and $y$ belong to the same chain $C_{i}$. Note that if $I\subset V(G)$ is an independent set, then $|I\cap C|\leq 1$ for $C\in\mathcal{C}$, so $|I|\leq M$. Therefore, $\alpha(G)=M$. It only remains to bound the number of edges of $G$. We are going to proceed similarly as in the proof of Corollary \ref{remark}. By the construction of $G$, we have
$$|E(G)|=\sum_{C\in\mathcal{C}}\binom{|C|}{2}\leq \frac{1}{2}\sum_{C\in \mathcal{C}}|C|^{2}.$$
Let $\mathcal{C}_{1}\subset \mathcal{C}$ be the family of chains $C$ such that $|C|\geq \sqrt{n}\log n$, and let $\mathcal{C}_{2}\subset \mathcal{C}$  be the family of chains $C$ such that $s+n^{1/2-1/20}<|C|< \sqrt{n}\log n$. Also, let $\mathcal{C}_{3}=\mathcal{C}\setminus (\mathcal{C}_{1}\cup\mathcal{C}_{2})$.

First, note that every chain of size $\sqrt{n}\log n$ must contain a set of size either at least $\frac{n+\sqrt{n}\log n}{2}$, or at most $\frac{n-\sqrt{n}\log n}{2}$. But by Claim~\ref{claim:binomial}, (3), we have $$\left|[n]^{(\leq \frac{n-\sqrt{n}\log n}{2})}\right|=\left|[n]^{(\geq \frac{n+\sqrt{n}\log n}{2})}\right|\leq 2^{n}e^{-(\log n)^{2}/2},$$ so $|\mathcal{C}_{1}|\leq 2^{n+1}e^{-(\log n)^{2}/2}$. Therefore,
$$\sum_{C\in\mathcal{C}_{1}}|C|^{2}\leq 2^{n+1}e^{-(\log n)^{2}/2}n^{2}=o(2^{n}).$$
Since, by Theorem \ref{thm:mainthm}, $|\mathcal{C}_{2}|<Mn^{-\frac{1}{8}+o(1)}$, we can write $$\sum_{C\in\mathcal{C}_{2}}|C|^{2}<Mn^{-\frac{1}{8}+o(1)}n(\log n)^{2}=o(2^{n}\sqrt{n}).$$
Finally, 
$$\sum_{C\in\mathcal{C}_{3}}|C|^{2}\leq M(s+n^{1/2-1/20})^{2}=Ms^{2}(1+o(1))=\left(\sqrt{\frac{\pi}{2}}+o(1)\right)2^{n}\sqrt{n}.$$
Therefore, $$|E(G)|\leq \frac{1}{2}\sum_{C\in\mathcal{C}_{1}}|C|^{2}+\frac{1}{2}\sum_{C\in\mathcal{C}_{2}}|C|^{2}+\frac{1}{2}\sum_{C\in\mathcal{C}_{3}}|C|^{2}\leq \left(\sqrt{\frac{\pi}{8}}+o(1)\right)2^{n}\sqrt{n},$$
finishing the proof.

\subsection{Applications to extremal problems}\label{sect:extremal}
Several problems in extremal set theory are instances of the following general question. Say that a formula is \emph{affine}, if it is built from variables, the operators $\cap$ and $\cup$, and parentheses $(\, ,)$ (complementation and constants are not allowed, e.g. $x\cap \{1,2,3\}$ and $x\setminus y$ are \emph{not} affine formulas.) Also, an \emph{affine statement} is a statement of the form $f\subset g$ or $f=g$, where $f$ and $g$ are affine formulas. Finally, an \emph{affine configuration} is a Boolean expression, which uses symbols $\lor, \land, \neg$ and whose variables are replaced with affine statements. Given an affine configuration $C$ with $k$ variables, a family $H\subset 2^{[n]}$ \emph{contains $C$}, if there exists $k$ distinct elements of $H$ that satisfy $C$, otherwise, say that \emph{$H$ avoids $C$}. Let $\ex(n,C)$ denote the size of the largest family $H\subset 2^{[n]}$ such that $H$ avoids $C$. Say that an affine configuration $C$ is \emph{satisfiable} if there exists a family of sets satisfying $C$.  

Here are some examples  of well known questions which ask to determine the order of magnitude of $\ex(n,C)$ for some specific affine configuration $C$.

\textbf{Sperner's theorem.} An antichain is exactly a family not containing the affine configuration $C\equiv(x\subset y)$. Hence, Sperner's theorem~\cite{S28} is equivalent to the statement $\mbox{ex}(n,C)=\binom{n}{\lfloor n/2\rfloor}$.

\textbf{Union-free families.} A family $H\subset 2^{[n]}$ is \emph{union-free}, if it does not contain three distinct sets $x,y,z$ such that $z=x\cup y$. But $H$ is union-free if and only if it does not contain the affine configuration $(z=x\cup y)$. The size of the largest union-free family was investigated by Kleitman~\cite{K76}, who proved that the size of such a family is at most $(1+o(1))\binom{n}{\lfloor n/2\rfloor}$. 

\textbf{Forbidden subposets.} Let $P$ be a poset, and let $\prec$ be the partial ordering on $P$. The following questions are extensively studied \cite{BJ12,B09,GL09, KT83, MP17, T19}: what is the maximum size of a family in $2^{[n]}$ that does not contain $P$ as a weak/induced subposet? For each $p\in P$, introduce the variable $x_{p}$. Then, forbidding $P$ as a weak subposet is equivalent to forbidding the affine configuration $$C_{P}\equiv\bigwedge_{\substack{p,q\in P\\ p\prec q}}(x_{q}\subset x_{q}),$$ while $P$ as an induced subposet corresponds to the affine configuration $$C'_{P}\equiv\bigwedge_{\substack{p,q\in P\\ p\prec q}}(x_{p}\subset x_{q})\wedge\bigwedge_{\substack{p,q\in P\\ p\not\prec q,q\not\prec p}}(\neg(x_{p}\subset x_{q})\wedge \neg(x_{q}\subset x_{p})).$$
Let $e(P)$ denote the maximum number $k$ such that the union of the $k$ middle levels of $2^{[n]}$ does not contain $C_{P}$, and define $e'(P)$ similarly for $C'_{P}$. It is commonly believed that $\ex(n,C_{P})=(e(P)+o(1))\binom{n}{\lfloor n/2\rfloor}$ and $\ex(n,C'_{P})=(e'(P)+o(1))\binom{n}{\lfloor n/2\rfloor}$. This conjecture has been only verified for posets with certain special structures, for example when the Hasse diagram of $P$ is a tree \cite{BJ12,B09}, so in general it is wide open. Also, while it is clear that $\ex(n,C_{P})\leq (|P|-1)\binom{n}{\lfloor n/2\rfloor}$ (as a chain of size $|P|$ satisfies $C_{P}$), it is already not obvious that $\ex(n,C'_{P})=O(\binom{n}{\lfloor n/2\rfloor})$. This was verified by Methuku and P\'alv\"olgyi \cite{MP17}. Finally, it is not even known whether the limit $\lim_{n\rightarrow\infty}\ex(n,C_{P})/\binom{n}{\lfloor n/2\rfloor}$ exists, see~\cite{GL09}.

\textbf{Boolean algebras.} The \emph{$d$-dimensional Boolean algebra} is a set of the form $$\{x_{0}\cup_{i\in I}x_{i}:I\subset [d]\},$$ where $x_{0},\dots,x_{d}$ are pairwise disjoint sets, $x_{1},\dots,x_{d}$ are nonempty. Let $b(n,d)$ denote the size of the largest family $H\subset 2^{[n]}$ that does not contain a $d$-dimensional Boolean algebra. It was proved by Erd\H{o}s and Kleitman~\cite{EK71} that $b(n,2)=\Theta(2^{n}n^{-1/4})$, where the constants hidden by the $\Theta(.)$ notation are unspecified and difficult to compute. Also, this was extended by Gunderson, R\"{o}dl and Sidorenko~\cite{GRS99} who proved that $b(n,d)=O(2^{n}n^{-1/2^{d}})$, where the constant hidden by the $O(.)$ notation depends on $d$. Finally, this was strengthened by Johnston, Lu and Milans \cite{JLM15} to $b(n,d)\leq 22\cdot 2^{n}n^{-1/2^{d}}$. Note that a Boolean algebra is equivalent to the following affine configuration: for $I\subset [d]$, let $x_{I}$ be a variable, then the corresponding affine configuration is $$\bigwedge_{1\leq i<j\leq d}(x_{\emptyset}=x_{\{i\}}\cap x_{\{j\}})\wedge\bigwedge_{I\subset [d], I\neq\emptyset} (x_{I}=\bigcup_{i\in I}x_{\{i\}}).$$
Moreover, the above results on Boolean algebras also show that for any formular  $C$, if it is satisfiable, then there exists $\alpha>0$ such that $\ex(n,C)=O(2^{n}n^{-\alpha})$. Indeed, if $C$ is satisfiable, then there exists $d$ such that $2^{[d]}$ contains $C$, but then every $d$-dimensional Boolean algebra also contains $C$.

\bigskip

Here, we provide a unified framework to handle such problems. First, let us consider a more general problem. A \emph{$d$-dimensional grid} is a $d$-term  Cartesian product of the form $[k_1]\times...\times[k_d]$, endowed with the following coordinatewise ordering $\subset$: $(a_{1},\dots,a_{d})\subset (b_{1},\dots,b_{d})$ if $a_{i}\leq b_{i}$ for $i=1,\dots,d$ (with slight abuse of notation, we also use $\subset$ to denote the comparability in the grid, for reasons that should become clear later). Also, define the operations $\cap$ and $\cup$ such that $(a_{1},\dots,a_{d})\cap (b_{1},\dots,b_{d})=(\min\{a_{1},b_{1}\},\dots,\min\{a_{d},b_{d}\})$ and $(a_{1},\dots,a_{d})\cup (b_{1},\dots,b_{d})=(\max\{a_{1},b_{1}\},\dots,\max\{a_{d},b_{d}\})$. Considering the natural isomorphism between the Boolean lattice $2^{[n]}$ and the grid $[2]^{n}$, $\subset,\cap,\cup$ naturally extend their usual definition. But now we can talk about affine configurations in the grid as well. If $F$ is a grid, say that a subset $H\subset F$ \emph{contains} the affine configuration $C$ with $k$ variables, if there exists $k$ distinct elements of $H$ that satisfy $C$, otherwise, say that \emph{$H$ avoids $C$}. Let $\ex(F,C)$ denote the size of the largest subset of $F$ which does not contain  $C$, and write $\ex(k,d,C)$ instead of $\ex([k]^d,C)$. 

Our aim is to show that one can derive bounds for $\ex(n,C)$ using the function $f(k)=\ex(k,d,C)$, where $d$ is some fixed integer. Indeed, by considering a chain decomposition of $2^{[n/d]}$ into chains of almost equal size, one can partition $2^{[n]}$ into $d$-dimensional grids that are also almost equal. Then, given a family $H\subset 2^{[n]}$ avoiding $C$, we bound the intersection of $H$ with each of these grids (using the function $\ex(k,d,C)$), which then turns into a bound on $\ex(n,C)$. The reason why we would like to work with $\ex(k,d,C)$ instead of $\ex(n,C)$ is that for many affine configurations $C$, estimating $\ex(k,d,C)$ is equivalent to an (ordered) hypergraph Tur\'an problem, which is sometimes easier to handle or already has good upper bounds.

Similar ideas were already present in~\cite{EK71,GRS99,MP17}, but executed in a somewhat suboptimal way. The following theorem is the main result of this section.

\begin{theorem}\label{thm:grid}
 Let $d$ be a positive integer, and let $c,\alpha>0$ such that $\mbox{ex}(k,d,C)\leq ck^{d-\alpha}$ holds for every sufficiently large $k\in\mathbb{Z}^{+}$. Then 
 $$\ex(n,C)\leq (1+o(1))c\left(\frac{2d}{\pi n}\right)^{\frac{\alpha}{2}}2^{n}.$$
\end{theorem}

Before we can prove this theorem, let us see how $\ex(F,C)$ and $\ex(k,d,C)$ are related.

\begin{claim}\label{claim:assymetric}
  Let $k\leq k_{1}\leq\dots\leq k_d$ and $F=[k_1]\times\dots\times[k_d]$. Then $$\frac{\ex(F,C)}{k_1\dots k_d}\leq \frac{\ex(k,d,C)}{k^d}.$$
\end{claim}

\begin{proof}
 Let $H\subset F$ such that $H$ does not contain a copy of $C$ and $|H|=\ex(F,C)$. For $i=1,\dots,d$, let $X_i$ be a random $k$ element subset of $[k_i]$, chosen from the uniform distribution, and let $F'=X_1\times\dots\times X_d$. Let $N=|F'\cap H|.$ Clearly, for every $v\in H$, we have $$\mathbb{P}(v\in F')=\frac{k^d}{k_1\dots k_d},$$
so $\mathbb{E}(N)=|H|\frac{k^d}{k_1\dots k_d}.$ Therefore, there exists a choice for $X_1,\dots,X_d$ such that $N\geq |H|\frac{k^d}{k_1\dots k_d}$. As $F'$ is isomorphic to the grid $[k]^d$ and $F'\cap H$ does not contain a copy of $C$, we get $$\ex(k,d,C)\geq \frac{k^d}{k_1\dots k_d}\ex(F,C).$$
\end{proof}

\begin{proof}[Proof of Theorem~\ref{thm:grid}]
In this proof, we consider $d$ as a constant, so the notation $O(.)$ hides a constant which might depend on $d$.

Let $H\subset 2^{[n]}$ be a subset of size $\mbox{ex}(n,C)$ not containing a copy of $C$. Write $n=n_1+\dots+n_d$, where $n_i\in\{\lfloor n/d\rfloor,\lceil n/d\rceil\}$ for $i=1,\dots,d$. Let $\mathcal{C}_i$ be a chain decomposition of $2^{[n_i]}$ given by Theorem~\ref{thm:mainthm}, that is, all but at most $n^{-\frac{1}{8}+o(1)}$ proportion of the chains in $\mathcal{D}_i$ have size $s(1+O(n^{-\frac{1}{16}}))$, where $$s=\left(\sqrt{\frac{\pi}{2}}+o(1)\right)\sqrt{\frac{n}{d}}.$$ If a chain $D\in \mathcal{C}_{i}$ is longer than $s(1+n^{-\frac{1}{20}})$, cut it into $\lceil\frac{|D|}{n}\rceil$ smaller chains such that the size of all but at most one of them is $s$. Let $\mathcal{D}_{i}$ be the resulting chain partition. As the number of chains of size more than $s(1+n^{-\frac{1}{20}})$ is at most $2^{n_{i}}n^{-\frac{5}{8}+o(1)}$, every chain in $\mathcal{D}_{i}$ has size at most $s(1+n^{-\frac{1}{20}})$, and the number of chains of size less than $s(1-n^{-\frac{1}{20}})$ is at most $2^{n_{i}}n^{-\frac{5}{8}+o(1)}$.

Let $\mathcal{D}=\{D_{1}\times\dots\times D_{d}:D_{1}\in\mathcal{D}_{1},\dots,D_{d}\in\mathcal{D}_{d}\}.$ Then $2^{[n]}$ is the disjoint union of the elements of $\mathcal{D}$.
Here, $D=D_1\times\dots\times D_{d}\in\mathcal{D}$ behaves exactly like the $d$-dimensional grid $F=[|D_1|]\times\dots\times[|D_d|]$. More precisely, let $\phi_i: D_i\rightarrow [|D_i|]$ be the bijection defined as $\phi_i(x_j)=j$, where $x_1\subset\dots\subset x_{|D_i|}$ are the elements of $D_i$. Setting $\phi=(\phi_1,\dots,\phi_d)$, $\phi$ is a bijection between $D$ and $F$ such that for any $x,y,z\in D$,
\begin{itemize}
\item $x\subset y$ if and only if $\phi(x)\subset \phi(y),$
\item $x\cup y=z$ if and only if $\phi(x)\cup\phi(y)=\phi(z),$
\item $x\cap y=z$ if and only if $\phi(x)\cap \phi(y)=\phi(z).$
\end{itemize}
But this means that a subset $H\cap D$ contains $C$ if and only if $\phi(H\cap D)$ contains $C$. Therefore, $|H\cap D|\leq \mbox{ex}(F,C)$.
Let $k=\min\{|D_{1}|,\dots,|D_{d}|\}$. Then by Claim~\ref{claim:assymetric}, we have 
\begin{equation}\label{equ:exFC}\ex(F,C)\leq \frac{|D_{1}|\dots
|D_{d}|}{k^{d}}\ex(k,d,C)\leq c|F|k^{-\alpha}\leq c(1+o(1))s^{d-\alpha}.
\end{equation}
Let $$\mathcal{D}'=\{D_{1}\times\dots\times D_{d}\in \mathcal{D}:|D_{i}|\leq s(1-n^{-\frac{1}{20}})\mbox{ for some }i\in [d]\}.$$ Then $|\mathcal{D}'|=o(2^{n}n^{-d/2}).$ Hence,  $$\sum_{D\in\mathcal{D}'}|D\cap H|\leq o(2^{n}n^{-\frac{d}{2}}s^{d-\alpha})=o(2^{n}n^{-\frac{\alpha}{2}}).$$
Also, by the second inequality in (\ref{equ:exFC}), we have 
$$\sum_{D\in\mathcal{D}\setminus \mathcal{D}'}|D\cap H|\leq \sum_{D\in\mathcal{D}\setminus \mathcal{D}'} (1+o(1)c|D|s^{-\alpha}\leq (1+o(1))c2^{n}s^{-\alpha}.$$
Therefore,
$$|H|=\sum_{D\in\mathcal{D}}|D\cap H|\leq (1+o(1))c\left(\frac{2d}{\pi n}\right)^{\frac{\alpha}{2}}2^{n}.$$
\end{proof}

Let us see some quick applications. Note that most of these applications were already covered in \cite{T19} with slightly worse constants.

\textbf{Sperner's theorem.} As an easy exercise, let us recover the asymptotic version of Sperner's theorem from Theorem \ref{thm:grid}. Indeed, let $C\equiv (x\subset y)$, then trivially $\ex(k,1,C)=1$. Therefore, $\ex(n,C)\leq (1+o(1))\sqrt{\frac{2}{n\pi}}2^{n}$.
    
\textbf{Union-free families.} Let $C\equiv (z=x\cup y)$. Consider the case $d=2$, then the affine configuration $C$ in $[k]^{2}$ corresponds to three points of the grid which form a corner, i.e., $(a,b), (c,b), (c,d)$ such that $a<c$ and $d<b$. 
It is not difficult to see that $\ex(k,2,C) \leq 2k$. Indeed, suppose $Q$ is a subset of the grid of order at least $2k + 1$. On every horizontal line delete the left most point and 
on every vertical line delete the lowest point which is in $Q$. Since we delete at most $2k$ points, some point $(b,c) \in Q$ must remain. Then, by definition, there are points $(a, b)$ and $(c,d)$ with $a < c$ and $d < b$ which are also in $Q$. 
Thus by Theorem \ref{thm:grid} (with $d=2$ and $\alpha=1$), we get $$\ex(n,C)\leq (1+o(1))2\sqrt{\frac{4}{\pi n}}2^{n}=(1+o(1))2\sqrt{2}\binom{n}{\lfloor n/2\rfloor},$$ which is only slightly worse than the bound of Kleitman \cite{K76}.
    
\textbf{Forbidden subposets.} Let $P$ be a poset that is not an antichain, and consider the corresponding affine configurations $C_{P}$ and $C'_{P}$. Let $d_{0}$ be the Duschnik-Miller dimension of $P$, that is, $d_{0}$ is the smallest $d$ such that $[k]^{d}$ contains the affine configuration $C'_{P}$ for some $k$. It was proved by Tomon \cite{T19} that there exists a constant $\alpha(P)$ such that if $d\geq d_{0}$, then $\ex(k,d,C'_{P})\leq \alpha(P)w$, where $w$ is the size of the largest antichain in $[k]^{d}$. We remark that $w=(1+o(1))\sqrt{\frac{6}{\pi}}\cdot\frac{k^{d-1}}{\sqrt{d}}$ as $\min\{k,d\}\rightarrow\infty$, see p.~63--68 in \cite{A87}. Let $$\beta(d,P)=\limsup_{k\rightarrow\infty} \frac{\sqrt{d}}{k^{d-1}}\ex(k,d,C_{P}),$$
and 
$$\beta'(d,P)=\limsup_{k\rightarrow\infty} \frac{\sqrt{d}}{k^{d-1}}\ex(k,d,C'_{P}).$$
Then $\beta(d,P)\leq \beta'(d,P)<\infty.$ Applying Theorem \ref{thm:grid}, we get $$\ex(n,C_{P})\leq (1+o(1))\beta(d,P)\binom{n}{\lfloor n/2\rfloor},$$ and $$\ex(n,C'_{P})\leq (1+o(1))\beta'(d,P)\binom{n}{\lfloor n/2\rfloor}.$$ This tells us that one can derive bounds on $\ex(n,C_{P})$ and $\ex(n,C'_{P})$ by considering the behavior of the functions $\ex(k,d,C_{P})$ and $\ex(k,d,C'_{P})$ for some fixed $d$. However, finding the values of these functions is equivalent to a forbidden $d$-dimensional matrix pattern problem (see e.g. \cite{KM06} for a description of this problem, and \cite{MP17,T19} for the connection of posets and matrix patterns), which provides us with new tools in order to estimate $\ex(n,C_{P})$ and $\ex(n,C'_{P})$.
    
\textbf{Boolean algebras.} Finally, let us consider Boolean algebras, in particular the case $d=2$. If $C$ is the affine configuration corresponding to the $2$-dimensional Boolean algebra, then a set $H\subset [k]\times [l]$ avoids $C$ if and only if $H$ does not contain four distinct points $(a,b),(a',b),(a,b'),(a',b')$, which is equivalent to a cycle of length four in the appropriate bipartite graph. But then by the K\H{o}v\'ari-S\'os-Tur\'an theorem \cite{KST54}, we have  
$$|H|\leq kl^{1/2}+O(k+l),$$ so $\ex(k,2,C)\leq (1+o(1))k^{3/2}.$ But then by Theorem \ref{thm:grid} (with $d=2$ and $\alpha=1/2$), we get  $$b(n,2)\leq (1+o(1))\left(\frac{4}{\pi n}\right)^{1/4}2^{n}.$$ One can get an even better bound by slightly modifying the proof of Theorem \ref{thm:grid}: instead of choosing $n_{1}=\lfloor n/2\rfloor$ and $n_{2}=\lceil n/2\rceil$, set $n_{1}=\lfloor n^{2/3}\rfloor$ and $n_{2}=n-n_{1}$, and write $\ex(F,C)\leq |D_{1}||D_{2}|^{1/2}+O(|D_{1}|+|D_{2}|)$. Then, after repeating the same calculations, we get $$b(n,2)\leq (1+o(1))\left(\frac{2}{\pi n}\right)^{1/4}2^{n}\approx 0.89\cdot 2^{n}n^{-\frac{1}{4}}.$$
We omit the details.

\section{Concluding remarks}\label{sect:remarks}

Let $M=\binom{n}{\lfloor n/2\rfloor}$, and let $\sigma_{1}\geq\dots\geq\sigma_{M}$ be the sizes of the chains in a symmetric chain decomposition of $2^{[n]}$. Let $D_{1},\dots,D_{M}$ be a chain decomposition of $2^{[n]}$ such that $|D_{1}|\geq \dots\geq |D_{M}|$. Then it is easy to show that the sequence $\sigma_{1},\dots,\sigma_{M}$ \emph{dominates} $|D_{1}|,\dots,|D_{M}|$, that is, $$\sum_{i=1}^{k}\sigma_{i}\geq \sum_{i=1}^{k} |D_{i}|$$ for $k=1,\dots,M$. Griggs \cite{G88} proposed the following conjecture.
\begin{conjecture}\label{conj:griggs}
Let $s_{1}\geq \dots\geq s_{M}$ be a sequence of positive integers dominated by $\sigma_{1},\dots,\sigma_{M}$ such that $\sum_{i=1}^{M}s_{i}=2^{n}$. Then there exists a chain decomposition $D_{1},\dots,D_{M}$ of $2^{[n]}$ such that $|D_{i}|=s_{i}$.
\end{conjecture}

Note that Conjecture \ref{conj:mainconj} is a special subcase of this conjecture, possibly the most
challenging one. One might consider a similar question for the upper half of $2^{[n]}$, that is, for the family $B=[n]^{(\geq n/2)}$. Then a conjecture akin to Conjecture \ref{conj:griggs} would be as follows. For $i=1,\dots,M$, let $\sigma_{i}'=\lceil \frac{\sigma_{i}}{2}\rceil$, then $\sigma_{1}',\dots,\sigma_{M}'$ are the sizes of the chains in a symmetric chain decomposition of $2^{[n]}$ restricted to $B$. 

\begin{conjecture}\label{conj:B}
Let $s_{1}\geq \dots\geq s_{M}$ be a sequence of positive integers dominated by $\sigma_{1}',\dots,\sigma_{M}'$ such that $\sum_{i=1}^{M}s_{i}=|B|$. Then there exists a chain decomposition $D_{1},\dots,D_{M}$ of $B$ such that $|D_{i}|=s_{i}$.
\end{conjecture}

It is plausible that one can use a modification of our approach to prove an asymptotic version of this conjecture.
I.e., there exists a chain decomposition $D_{1},\dots,D_{M}$ of $B$ such that for all but at most $o(M)$ indices $i\in [M]$, we have $|D_{i}|=(1+o(1))s_{i}$. However, such a result will not immediately yield an asymptotic version of Conjecture \ref{conj:griggs} for the following reason: we might be able to partition the lower and upper half of $2^{[n]}$ into chains of the desired lengths, but when we try to match the chains in the lower and upper half, we are unable to guarantee that the chains of right lengths are connected.

\begin{appendices}
\section{Appendix: the proof of Claim \ref{claim:binomial} and Claim~\ref{claim:calc}}

\begin{proof}[Proof of Claim \ref{claim:binomial}]
(4): For $0.1\sqrt{n}\leq l <\sqrt{n}$ this follows from (2). For $l<0.1\sqrt{n}$, we have the upper bound $$\binom{n}{m+l}=M\prod_{j=1}^{l}\frac{n-m-l+j}{m+j}\leq M\left(1-\frac{l}{n}\right)^{l}\leq Me^{-l^2/n}\leq M\left(1-\frac{l^{2}}{4n}\right).$$
  For the lower bound we proceed in a similar fashion:
  $$\binom{n}{m+l}=M\prod_{j=1}^{l}\frac{n-m-l+j}{m+j}\geq M\left(1-\frac{2l}{n}\right)^{l}\geq M\left(1-\frac{2l^{2}}{n}\right).$$
  
  (5): $$\binom{n}{m+l}-\binom{n}{m+l+1}=\binom{n}{m+l}\left( 1-\frac{n-(m+l)}{m+l+1}\right)=\binom{n}{m+l}\frac{2m-n+1+2l}{m+l+1}.$$
   Here, $\binom{n}{m+l}=\Theta(2^{n}n^{-1/2})$ by (1) and (2), and $\frac{2m-n+1+2l}{m+l+1}=\Theta(\frac{l}{n})$. Therefore,
   $$\binom{n}{m+l}-\binom{n}{m+l+1}=\Theta(l2^{n}n^{-3/2}).$$
   
   (6): Using part (2), we have that $$\frac{\binom{n}{m+l}}{\binom{n}{m+l+n/l}}=(1+o(1))\frac{e^{-2l^2/n}}{e^{-\frac{2}{n}\left(l^2+2n + \frac{n^2}{l^2}\right)}}=(1+o(1))e^{4+\frac{2n}{l^2}}\leq (1+o(1))e^6.$$
   Hence, 
   $$\sum_{i\geq m+l}\binom{n}{i}\geq \sum_{i=m+l}^{m+l+\frac{n}{l}}\binom{n}{i}\geq (e^{-6}+o(1))\binom{n}{m+l}\frac{n}{l}=(e^{-6}+o(1))\frac{Mn}{l}e^{-2l^2/n},$$
   and the result follows by applying (1) and using that $\sqrt{\frac{2}{\pi}}\geq \frac{1}{e}$.
\end{proof}

\begin{proof}[Proof of Claim \ref{claim:calc}]
We omit floors and ceiling for simplicity. Let $1\leq a \leq k$ and $a\leq b \leq k$. We shall prove that $\phi(a,b)$ and $\phi(a,b+1)$ are disjoint and Claim~\ref{claim:calc} will follow. Equivalently, we will show that the sum of the lengths of the two partial diagonals between $X_{a,b}$ and $X_{a,b+1}$ is longer than the corresponding layer. Suppose for contradiction that $(a,b)$ is the first element of $I$ such that there exists some $r\leq \left\lceil \sqrt{\frac{1}{3}\log n} \right\rceil$  with $r\sqrt{n}+1\in \phi(a,b)\cap \phi(a,b+1)$. 

Suppose first that $b-a \leq 0.01\sqrt{n}$. Recall that the length of the $j$-th diagonal is $M-|A_{k-j}|$. Then by Claim~\ref{claim:binomial}, (2), each of the first $b-a$ diagonals have size at least $(1+o(1)) M\left(1-e^{-(\sqrt{\pi/8} - 0.01)^2}\right)>0.51M$. On the other hand, for any $i\geq k$ we have $|A_i|\leq |A_k| = (1+o(1))M\left(1-e^{-\pi/4}\right)\leq 0.48M$. Hence, $\phi(a,b)\cap \phi(a,b+1)=\emptyset$ in this range. Next we will assume $ b-a\geq 0.01\sqrt{n}$ and so in particular we have $r\geq \frac{k}{\sqrt{n}}+ 0.01$.

We will use Claim~\ref{claim:binomial}, (2) to approximate the size of unions of layers  by integrals, which will give an approximation up to a $(1+o(1))$ factor. As all inequalities we wish to show in this proof are far from sharp, this error term will not cause any problems for us.
Set $k'=\frac{k}{\sqrt{n}}=(1+o(1))\sqrt{\frac{\pi}{8}}$, and let
\begin{equation}\label{eq:1}
S:=\sum_{i=k+1}^{r\sqrt{n}}|A_i|=(1+o(1))M\int_{k+1}^{r\sqrt{n}}e^{-2x^2/n}dx = (1+o(1))M\sqrt{n}\int_{k'}^re^{-2x^2}dx,
\end{equation}
that is, $S$ is the total number of elements in the levels $A_{k+1},\dots,A_{r\sqrt{n}}$. But $(a,b)$ is the first element of $I$ with $r\sqrt{n}+1\in \phi(a,b)$, so  we have $S_{0}\leq S\leq S_{1}$, where $S_{0}=\sum_{(a',b')\prec (a,b)}|X_{a',b'}|$ and $|S_{1}|=|S_{0}|+|X_{a,b}|$. In particular, $|S_{0}|(1+o(1))=|S|$. Next, we will calculate the size of $S_0$ by summing up the first $b-a-1$ diagonals and adding to it the piece of the $(b-a)$-th diagonal that comes before $X_{a,b}$ -- see Figure~\ref{figure:chains}. Let $t=b-a-1$ and $t'=\frac{t}{\sqrt{n}}$, then we have

    \begin{equation}\label{eq:2}
    \begin{split}
S_{0}&= \sum_{i\prec (a,b)} |X_i| = \left(M-|A_{a-1}| + \sum_{i=0}^{b-a-1}\left(M-|A_{k-i}|\right)\right)\\
&= (1+o(1))\left(t'\sqrt{n}M - M\sqrt{n}\int_{k'-t'}^{k'} e^{-2x^2}dx \right).
\end{split}
\end{equation}
By comparing~(\ref{eq:1}) and~(\ref{eq:2}), we get
$$(1+o(1))\int_{k'}^re^{-2x^2}dx = t' - \int_{k'-t'}^{k'} e^{-2x^2}dx,$$
which implies, that
\begin{equation}\label{eq:3}
t'=\int_{k'-t'}^{r}e^{-2x^{2}}dx + o(1).
\end{equation}
For this $t$ and $r$, we wish to show that the level $m+r\sqrt{n}$ is smaller than the $t$-th diagonal, that is,
\begin{equation}\label{eq:x}
M-\binom{n}{m+(k'-t')\sqrt{n}}>\binom{n}{m+r\sqrt{n}}.
\end{equation}
Using by Claim \ref{claim:binomial}, (2), the left hand side is equal to $(1+o(1))M\left(1-e^{-2(k'-t')^{2}}\right)$, and the right hand side is equal to $(1+o(1))Me^{-2r^{2}}$. Hence to establish~(\ref{eq:x}) for large $n$, it suffices to show that 
\begin{equation}\label{eq:des}
1-e^{-2(k'-t')^{2}}>1.01\cdot e^{-2r^{2}}
\end{equation}

By ~(\ref{eq:3}), we can view $t'$ as a function of $r$. Let $f:(\sqrt{\frac{\pi}{8}},\infty)\rightarrow \mathbb{R}^{+}$ be the function satisfying $$f(r)=\int_{\sqrt{\frac{\pi}{8}}-f(r)}^{r}e^{-2x^{2}}dx,$$ then $f$ is well defined, strictly increasing and continuous.  Note that $$\frac{df}{dr}=\frac{e^{-2r^2}}{1-e^{-2(\sqrt{\frac{\pi}{8}}-f(r))^2}},$$ and that $f$ and $\frac{df}{dr}$  are both absolutely continuous on the interval $[\sqrt{\frac{\pi}{8}},4]$, say. Using elementary methods,we will now show that we have $f(r)=(1+o(1))t'$ in the range of parameters where $t'\geq 0.01$ and $r\leq 4$. Indeed note that in this range there exists a positive constant $\delta>0$ such that for any $x$ in the domain of the integral, we have that $\delta < e^{-2x^2}<1-\delta$. Assume first that the $o(1)$ term in equation~(\ref{eq:3}) is positive, so that $t'=\int_{k'-t'}^{r}e^{-2x^{2}}dx + \epsilon$. Then letting $t'' := t' - \frac{4\epsilon}{\delta}$ we get that $t'' \leq \int_{k'-t''}^{r}e^{-2x^{2}}dx -\epsilon$ as the left hand side decreased by $\frac{4\epsilon}{\delta}$ and the right hand side decreased by at most $\frac{4\epsilon}{\delta}\left(1-\frac{\delta}{2}\right)$, so the cumulative drop was at least $2\epsilon$. By continuity, there is a value of $t^\ast$ between $t''$ and $t'$ satisfying $t^\ast = \int_{\sqrt{\frac{\pi}{8}}-t^\ast}^{r}e^{-2x^{2}}dx$. Hence in this case we have $t'-\frac{2\epsilon}{\delta}\leq f(r)\leq t'$ and indeed $t'-\frac{2\epsilon}{\delta}=t'(1-o(1))$ -- the other case where the $o(1)$ term in equation~(\ref{eq:3}) is negative is very similar.

We conclude that our desired inequality~(\ref{eq:des}) for $r\leq 4$ would be a consequence of the inequality $\frac{df}{dr}<0.98$ for all $r\in [\sqrt{\frac{\pi}{8}},4]$. Note that for $r=\sqrt{\frac{\pi}{8}}$ this holds as $2e^{-\pi/4}<0.98$. One way to prove that $\frac{df}{dr}<0.98$ for all $r$ is to show that the second derivative of $f$ with respect to $r$ is negative and hence $\frac{df}{dr}$ is decreasing. A shorter proof, which we will present here, uses numerical methods to verify that $\frac{df}{dr}<0.98$ for $r\leq 3.8$ and uses a different approach to handle the large $r$ case.

Let $L(r):=e^{-2r^2}$ and $R(r):=0.98\left(1-e^{-2(\sqrt{\frac{\pi}{8}}-f(r))^2}\right)$. Then we wish to show that $L(r)<R(r)$ for all $r\in [\sqrt{\frac{\pi}{8}},3.8]$. Observe that since $f(r)$ is strictly increasing, both $L(r)$ and $R(r)$ are decreasing functions of $r$. Our strategy will be to make use of the fact that the inequality $L(r)<R(r)$ is far from sharp in this range. We will find reals $\sqrt{\frac{\pi}{8}}=r_0<r_1<\ldots < r_s$ for some integer $s$, such that $L(r_{i})<R(r_{i+1})$ for all $i$, and $r_s>3.8$. This will then imply that $L(r)<R(r)$ for all $r\leq 3.8$. Indeed, for any $r\leq 3.8$ we find $i$ such that $r_i\leq r \leq r_{i+1}$ and then we have $L(r)\leq L(r_i)<R(r_{i+1})\leq R(r)$ and so $L(r)<R(r)$. A  list of such reals is given in the table below.
\begin{center}
\begin{tabular}{ c c c c}
 $i$ & $r_i$ & $L(r_i)$ & $R(r_i)$  \\ 
 0 & $\sqrt{\frac{\pi}{8}}$ & 0.4559 & ~ \\
 1 & 0.709375 & 0.3655 & 0.4653 \\
 2 & 0.809451 & 0.2697 & 0.3742 \\
 3 & 0.928680 & 0.1781 & 0.2771 \\
 4 & 1.069430 & 0.1015 & 0.1838 \\
 5 & 1.235140 & 0.0473 & 0.1052 \\
 6 & 1.430872 & 0.01666 & 0.04931 \\
 7 & 1.663845 & 0.003939 & 0.01747 \\
 8 & 1.943875 & 0.0005222 & 0.004161 \\
 9 & 2.283642 & $2.953\cdot 10^{-5}$ & $5.566\cdot 10^{-4}$ \\
10 & 2.698861 & $4.713\cdot 10^{-7}$ & $3.181\cdot 10^{-5}$ \\
11 & 3.208593 & $1.142\cdot 10^{-9}$ & $5.145\cdot 10^{-7}$ \\
12 & 3.835987 & ~ & $1.27\cdot 10^{-9}$ \\
\end{tabular}
\end{center}

It remains to handle the case where $r\geq 3.8$. Let $u:=k-(b-a)$.  Then by counting the elements in the levels above $A_{r\sqrt{n}}$ and using Claim~\ref{claim:binomial}, (6), we get that for $n$ large enough,
$$Q:=\sum_{i> r\sqrt{n}} |A_i|\geq \left(e^{-7}+o(1)\right)2^n e^{-2r^2}\frac{\sqrt{n}}{r\sqrt{n}} \geq 2^ne^{-2r^2}\frac{e^{-8}}{r}.$$

Recall that $(a,b)$ is the first element of $I$ such that $r\sqrt{n}+1\in \phi(a,b)$. Hence $Q$ is upper bounded by the total size of the last $u+1$ diagonals, i.e.~
$$Q\leq \sum_{i=0}^{u}\left(M - \binom{n}{m + i}\right)\leq \frac{2M}{n}\sum_{i=1}^{u}i^2\leq \frac{2u^3M}{n},$$
where the second inequality follows from Claim~\ref{claim:binomial}, (4). Putting the last two inequalities together, we get $\frac{2u^3M}{n}\geq 2^ne^{-2r^2}\frac{e^{-8}}{r}$. Using that $\frac{2^n}{M}=(\sqrt{\frac{\pi}{2}}+o(1))\sqrt{n}\approx 1.25\sqrt{n}$ and that $e^8\approx 2981$, and taking the $2/3$-th power of both sides, we have
\begin{equation}\label{eq:app1}
\frac{u^2}{n}\geq \frac{e^{-4r^2/3}}{300 r^{2/3}}.
\end{equation}
The size of the $(k-u)$-th diagonal is 
$$M-\binom{n}{m+u}\geq M\frac{u^2}{4n}.$$
The size of the $m+r\sqrt{n}+1$-th layer, i.e.~$A_{r\sqrt{n}+1}$, is
$$|A_{r\sqrt{n}+1}|=(1+o(1))Me^{-2r^2}.$$
If the $r\sqrt{n}+1$-th layer was indeed larger than the $(k-u)$-th diagonal then we would have, for large enough $n$, that
$5e^{-2r^2}\geq \frac{u^2}{n}.$ However,~(\ref{eq:app1}) implies that
$$\frac{u^2}{n}\geq \frac{e^{-4r^2/3}}{300r^{2/3}}\geq 5.1e^{-2r^2},$$
where the last inequality holds for $r\geq 3.6$. As $r\geq 3.8$, this finishes the proof.
\end{proof} 

\end{appendices}

\end{document}